\documentclass[journal]{IEEEtran}
\usepackage{cite}

\ifCLASSINFOpdf
  \usepackage[pdftex]{graphicx}
  \graphicspath{./graphics/}
  \DeclareGraphicsExtensions{.pdf,.jpeg,.png}
\else
  \usepackage[dvips]{graphicx}
  \graphicspath{../graphics/}
  \DeclareGraphicsExtensions{.eps}
\fi
\usepackage[cmex10]{amsmath}
\usepackage{algorithmic}
\usepackage[tight,footnotesize]{subfigure}

\usepackage{amssymb}  
\usepackage{amsthm}
\usepackage{algorithm} 
\usepackage{color}
\usepackage{enumerate}
\usepackage[textwidth=2cm,textsize=small]{todonotes}
\usepackage{flushend}  

\theoremstyle{plain}
\newtheorem{thm}{Theorem}
\newtheorem{prop}{Proposition}
\newtheorem{lem}{Lemma}
\newtheorem{cor}{Corollary}
\newtheorem{assumption}{Assumption}

\theoremstyle{definition}
\newtheorem{definition}{Definition}

\theoremstyle{remark}

\newtheorem{remark}{Remark}

\newcommand{\be}{\begin{enumerate}}
\newcommand{\ee}{\end{enumerate}}

\newcommand{\ip}[2]{\left \langle #1 , #2 \right \rangle}

\begin{document}
%
\title{Adaptive Algorithms for Coverage Control and Space Partitioning in Mobile Robotic Networks}
%
%
%

\author{Jerome~Le~Ny,~\IEEEmembership{Member,~IEEE,}
        and~George~J.~Pappas,~\IEEEmembership{Fellow,~IEEE}
\thanks{Long version of a manuscript to appear in the Transactions on Automatic Control.
Submitted November 4, 2010; revised November 30, 2011, and July 2012.
This work was supported by the ONR-MURI HUNT award N00014-08-1-0696.
}%
\thanks{J. {Le Ny} is with the Department of Electrical Engineering, \'Ecole Polytechnique
de Montr\'eal, QC H3T 1J4, Canada (email: {\tt\small jerome.le-ny@polymtl.ca}). 
G. Pappas is with the Department of Electrical and Systems Engineering, 
University of Pennsylvania, Philadelphia, PA 19104, USA 
(email: {\tt\small pappasg@seas.upenn.edu}).}
}

%
%

\markboth{University of Pennsylvania, ESE Technical Report}%
{Le Ny and Pappas: Adaptive Deployment}
%



\maketitle

\begin{abstract}
This paper considers deployment problems where a mobile robotic network must optimize its configuration in a distributed way 
in order to minimize a steady-state cost function that depends on the spatial distribution of certain probabilistic events of interest. 
Moreover, it is assumed that the event location distribution is a priori unknown, and can only be progressively inferred 
from the observation of the actual event occurrences. 
Three classes of problems are discussed in detail: coverage control problems, spatial partitioning problems, 
and dynamic vehicle routing problems. 
In each case, distributed stochastic gradient algorithms optimizing the performance objective are presented.
The stochastic gradient view simplifies and generalizes previously proposed solutions, and is applicable to new complex scenarios, 
such as adaptive coverage involving heterogeneous agents. 
Remarkably, these algorithms often take the form of simple distributed rules that could be implemented on resource-limited platforms.
\end{abstract}


%
\IEEEpeerreviewmaketitle

\section{Introduction}

The deployment of large-scale mobile robotic networks has been an actively investigated topic 
in recent years \cite{Howard02_deployment, Ogren04_multirobot, Bullo09_book}. 
Applications range from intelligence, surveillance and reconnaissance missions with unmanned aerial vehicles 
to environmental monitoring, search and rescue missions, and transportation and distribution tasks. 
With the increase in size of these networks, relying on human operators to remotely pilot each vehicle is becoming impractical. 
Attention is increasingly focusing on enabling autonomous operations, 
so that these systems can decide online how to concentrate their activities where they are most critical.

A mobile robotic network should have the capability of autonomously deploying itself in a region of interest 
to reach a configuration optimizing a given performance objective \cite[Chapter 5]{Bullo09_book}. 
Such problems can be distinguished based on the deployment objective, and among them 
the coverage control problem introduced by Cort{\'e}s et al. \cite{cortesTRO04_coverage} has proved to be particularly important. 
In this problem, the quality of a given robot configuration is measured by a multicenter function from the locational optimization 
and vector quantization literature \cite{Gray84_vectorQuantization, Drezner95_facilityLocation}. 
A distributed version of the Lloyd quantization algorithm \cite{Lloyd82_Quantization} allows a robotic network to locally optimize 
the utility function in a way that scales gracefully with the size of the network \cite{cortesTRO04_coverage}. 
The basic version of the coverage control problem has inspired many variations, 
e.g., considering limited communication and sensing radii \cite{Cortes05_coverageComm, Li05_coverageControl}, 
heterogeneous sensors \cite{Guruprasad09_heterogeneousCoverage}, 
obstacles and non-point robots \cite{Pimenta08_NonConvexCoverage}, 
or applications to field estimation problems \cite{Schwager09_adaptiveCoverage}.
It is also related to certain vehicle routing problems, 
notably the Dynamic Traveling Repairman Problem (DTRP) \cite{Bertsimas91_DTRP, Bertsimas93_multiDTRP, Bertsimas93_DTRP}, 
as discussed by Frazzoli and Bullo in \cite{Frazzoli04_DTRP} 
and several subsequent papers \cite{Savla07_thesis, Pavone10_adaptiveDTRP}. 
Another related problem is the space partitioning problem \cite{Pavone09_equipartitions, Cortes10_loadBalancing}, 
where the robots must autonomously divide the environment in order to balance the workload among themselves.

In essentially all the previously mentioned applications, the goal of the robotic network 
is to respond to events appearing in the environment. 
For example in the DTRP, jobs appear over time at random spatial locations 
and are serviced by the mobile robots traveling to these locations. 
The utility function optimized by the network invariably depends on the spatial distribution of the events, 
and the optimization algorithms require the knowledge of 
this distribution \cite{cortesTRO04_coverage, Frazzoli04_DTRP, Pavone09_equipartitions, Cortes10_loadBalancing}. 
Hence they are not applicable in the commonly encountered situations where the robots do not initially have such knowledge 
but can only observe the event locations over time. It is then natural to ask how to gradually improve the spatial configuration 
of the robotic network based only on these observations.
Indeed, recently some coverage control algorithms \cite{Choi09_coverage, Schwager09_adaptiveCoverage} and 
vehicle routing algorithms \cite{Arsie09_routing, Pavone10_adaptiveDTRP} have been developed to 
work in the absence of a priori knowledge of the event location distribution. 

An essential idea of our work is that deployment problems with stochastic uncertainty can be discussed 
from the unifying point of view of stochastic gradient algorithms, thereby clarifying the convergence proofs 
and allowing to easily derive new algorithms for complex problems. 
In this paper we restrict our attention to three related classes of problems: coverage control, spatial partitioning, 
and dynamic vehicle routing problems. For these three applications, we derive distributed stochastic gradient algorithms 
that optimize the utility functions \emph{in the absence of a priori knowledge of the event location distribution}.
We call these algorithms \emph{adaptive}, in analogy with the engineering literature on adaptive systems \cite{Benveniste90_SA}.
Remarkably, the algorithms we describe often take the form of simple rules, in fact typically simpler than the corresponding non-adaptive algorithms. 
Hence they are easier to implement on small platforms with limited sensing, computational and communication capabilities.

Specifically, we first discuss in Section \ref{section: coverage problems} certain 
stochastic gradient algorithms that adaptively optimize coverage control objectives. 
These algorithms generalize to new complex multi-agent deployment problems and 
we justify this claim by developing solutions to coverage control problems 
involving Markovian event dynamics or heterogeneous robots. 
Additional application examples, including deployment under 
realistic stochastic wireless connectivity constraints, 
can be found in \cite{LeNy10_adaptiveRobotsTR, LeNy12_adaptiveComm}. 
In Section \ref{section: load-balancing}, we describe new adaptive distributed algorithms 
that partition the workspace between the robots in order to balance their workload, 
using only the observation of the past event locations. 
These algorithms exploit the link between generalized Voronoi diagrams and 
certain Monge-Kantorovich optimal transportation problems 
\cite{Rachev98_optimalTransport, Gangbo96_optimalTransport, Ruschendorf97_coupling}. 
Finally in Section \ref{section: vehicle routing problems} we present the first 
fully adaptive algorithm for the DTRP. 
In light traffic conditions, the policy reduces to the 
coverage control algorithm of Section \ref{section: coverage problems}, 
and is simpler than the previous algorithm presented in \cite{Arsie09_routing}. 
In heavy traffic conditions, it relies on the partitioning algorithm 
of Section \ref{section: load-balancing}. 
This algorithm complements the recent work of Pavone et al. \cite{Pavone10_adaptiveDTRP}, 
in which the knowledge of the event location distribution is required in the heavy traffic regime.


\section{Preliminaries}	\label{section: preliminaries}

\subsection{Notation}	\label{section: notations}

We denote $[n] := \{1,\ldots,n\}$.  
Throughout the paper all random elements are defined on a generic probability space $(\Omega, \mathcal F, P)$, 
with the expectation operator corresponding to $P$ denoted $E$. 
We abbreviate ``independent and identically distributed'' by iid, and ``almost surely'' by a.s.
We denote the \emph{Euclidean norm} on $\mathbb R^q$ by $\| \cdot \|$.

Let $(X,d)$ be a metric space.
For a set $S \subset X$, we denote the indicator function of $S$ by  $\mathbf{1}_{S}$, 
i.e., $\mathbf 1_S(x) = 1$ if $x \in S$ and $\mathbf 1_S(x) = 0$ otherwise.
For $x_0 \in X$, the Dirac measure at $x_0$ is denoted by $\delta_{x_0}$ and 
defined by $\delta_{x_0}(S) = \mathbf 1_S(x_0)$ for all Borel subsets $S$ of $X$. 
We denote the distance from a point $x \in X$ to a set $S$ by
$
d(x,S) := \inf_{y \in S} d(x,y),
$
and we set $d(x,\emptyset) = + \infty$. 
A sequence of points $\{x_k\}_{k \geq 0}$ in $X$ is said to converge to 
a set $S \subset X$ if $d(x_k,S) \to 0$ as $k \to \infty$.
For nonempty sets $B,C \subset X$, the Hausdorff pseudometric is 
defined by $d_H(B,C) := \max(\sup_{x \in B} d(x,C), \sup_{x \in C} d(x,B))$.
The ball of radius $r$ around $S \subset X$ is
$
B(S,r) := \{x \in X | d(x,S) \leq r\}.
$
Also, $B(\{x\},r)$ is just denoted $B(x,r)$. 

A solution of a differential equation $\dot x = h(x)$ or of a differential inclusion $\dot x \in \mathcal H(x)$ \cite{Cortes08_discontinuous} is
interpreted in the sense of Caratheodory, i.e., as an absolutely continuous function $x(t)$ satisfying 
\begin{align*}
x(t) = x_0 + \int_0^t y(s) \; ds, \text{ for all } t \in \mathbb R, \text{with } y(s)=h(x(s)) \\
\text{ or } y(s) \in \mathcal H(x(s)) \text{ for all } s.
\end{align*}
Finally, a set I is invariant under 
a differential inclusion $\dot x \in \mathcal H(x)$ if for all $x_0 \in I$, there exist \emph{some} solution
$x(t), t \in (-\infty,\infty)$, with $x(0) = x_0$, that lies entirely in $I$.

\subsection{Robot Network Model}		\label{section: model}

We consider a group of $n$ robots evolving in $\mathbb R^q$, 
for some $q \geq 1$. 
We denote the robot positions at time $t \in \mathbb R_{\geq 0}$ by $p(t) = [p_1(t),\ldots,p_n(t)] \in (\mathbb R^q)^n$. 
For simplicity, we assume that the robots follow a simple kinematic model
\begin{equation}	\label{eq: trivial dynamics continuous time}
\forall i \in [n], \forall t \in \mathbb R_{\geq 0}, \; \dot p_i(t) = u_i,\; |u_i(t)| \leq v_i,
\end{equation}
where $v_i$ is a positive constant and $u_i$ is a bounded control input.
However, more complex dynamics could be considered since our analysis only involves the positions of the robots at certain discrete times, 
see, e.g., (\ref{eq: trivial dynamics discrete time}).  In addition, the robots are assumed to perform computations and to communicate instantaneously. 
Finally, we define
\begin{align}	\label{eq: diagonal set definition}
\mathsf D_n = \Big \{p= [p_1,\ldots,p_n] \in & (\mathbb R^{q})^n \; \Big | \\
& p_i = p_j \text{ for some } 1 \leq i < j \leq n \Big \}. \nonumber
\end{align}
Hence $\mathsf D_n$ is the set of configurations where at least two robots occupy the same position.

\subsection{Geometric Optimization}

For a vector $p = [p_1,\ldots,p_n] \in (\mathbb R^{q})^n\setminus \mathsf D_n$, we define the Voronoi cell of the point $p_i$ by
\[
V_{i}(p) = \Big \{z \in \mathbb R^q \Big | \|z-p_{i}\| \leq \|z-p_{j}\|, \forall j \in [n] \Big \}.	
\]
That is, $V_{i}$ is the set of points for which robot $i$ is the closest robot \emph{for the Euclidean distance}. 
The Voronoi cells of the points divide $\mathbb R^q$ into closed convex polyhedra, and $\{V_i\}_{i \in [n]}$ 
is called a Voronoi diagram \cite{Okabe2000_Voronoi}. 
Two points $p_i$ and $p_j$ or their indices $i, j$ (with $i \neq j$) are called Voronoi neighbors 
if the boundaries of their Voronoi cells intersect, i.e., if $V_i(p) \cap V_j(p) \neq \emptyset$.

For a function $c: \mathbb R^q \times \mathbb R^q \to \mathbb R$, 
a vector $w = [w_1,\ldots,w_n] \in \mathbb R^n$, 
and $p = [p_1,\ldots,p_n] \in (\mathbb R^{q})^n\setminus \mathsf D_n$,
define for all $i \in [n]$ the generalized Voronoi cell of the pair $(p_i,w_i)$ with respect to $c$ by
\begin{align}	\label{eq: generalized Voronoi}
V^c_{i}(p,w) = \Big \{z \in \mathbb R^q \Big | c(z,p_{i}) - w_i \leq c(z,p_{j}) - w_j, \\ 
\forall j \in [n] \Big \}.	\nonumber
\end{align}
We also write $V_i^c(\mathcal G,w):=V^c_{i}(p,w)$ for the set $\mathcal G = \{p_1,\ldots,p_n\}$.
The point $p_i$ is called the generator and $w_i$ the weight of the cell $V^c_i(p,w)$, and $\{V^c_i\}_{i \in [n]}$ a generalized Voronoi diagram.
Intuitively $c(z,p)$ represents a distance or cost between the points $z$ and $p$, and in practice takes the form
$c(z,p) = f(\|z-p\|)$, with $f$ an increasing function.
In particular for $f(x) = x^2$, the generalized Voronoi diagram is called 
a power diagram \cite{Aurenhammer91_VoronoiSurvey, Okabe2000_Voronoi}, 
and the generalized Voronoi cell a power cell. Like Voronoi cells, power cells are polyhedra, 
although possibly empty \cite{Aurenhammer91_VoronoiSurvey}.
Notice from (\ref{eq: generalized Voronoi}) that for a given configuration $p$, the size of a generalized Voronoi cell 
of a pair increases as its weight increases with respect to the weights of the other pairs.
Similarly to Voronoi neighbors, we define generalized Voronoi neighbors and power diagram neighbors.

\subsection{Min-consensus}	\label{section: min-consensus}

At several occasions, we need to solve the following problem in a distributed manner in the robotic network. 
Robot $i$, for $i \in [n]$, is associated to a certain quantity $\hat d_i \in \mathbb R$, 
which can be $+\infty$. Each robot must decide if it belongs to the set $\arg \min_{i \in [n]} \hat d_i$.
For simplicity, we assume that each robot can communicate with some other robots along bidirectional links 
in such a way that the global communication network is connected. 
We also assume that the robots know the diameter of the network, denoted \texttt{diam}. 
Alternatively, they know the number $n$ of robots in the system, in which case we take $\texttt{diam} = n$ below.

In a synchronous network the problem can be solved by the \texttt{FloodMin} algorithm \cite[section 4.1.2]{Lynch1996distributed}. 
Each robot maintains a record in a variable $\mathtt{d_i}$ of the minimum number it has seen so far, 
with $\mathtt{d_i} = \hat d_i$ initially. At each round, it sends this minimum to all its neighbors. 
The algorithm terminates after \texttt{diam} rounds. The agents that still have $\mathtt{d_i} = \hat d_i$ at the end 
know that they belong to $\arg \min_{i \in [n]} \hat d_i$. This algorithm can also be implemented 
in an asynchronous network by adding round numbers to the transmitted messages \cite[section 15.2]{Lynch1996distributed}.


\section{Adaptive Coverage Control Algorithms}	\label{section: coverage problems}

\subsection{Coverage Control for Mobile Robotic Networks}	\label{section: coverage control review}

In the standard coverage control problem \cite{cortesTRO04_coverage}, the goal of the robotic network 
is to reach asymptotically a configuration where the agent positions 
$\lim_{t \to \infty} p_i(t), i \in [n],$ 
minimize the following performance measure capturing the quality of coverage of certain events:
\begin{align}
\mathcal E_n(p) &= E \left[\min_{i \in [n]}  f(\|p_i - Z\|) \right], 	\label{eq: coverage simple general}
\end{align}
where $f: \mathbb R_{\geq 0} \to \mathbb R_{\geq 0}$ is an increasing continuously differentiable function. 
The random variable $Z$ represents the location of an event of interest occurring in the workspace. 
To interpret (\ref{eq: coverage simple general}), the cost of servicing an event at location $z$ 
with a robot at location $p_i$ is measured by $f(\|p_i-z\|)$, and 
an event must be serviced by the robot closest to the location of this event.  
For example, in monitoring applications, 
$f(\|p_i-z\|)$ can measure the degradation of the sensing performance with the distance 
to the event \cite{cortesTRO04_coverage}. In vehicle routing problems, this cost might be the 
time it takes a robot to travel to the event location, i.e., $f(\|p_i-z\|) = \|p_i-z\|/v_i$, assuming enough time between
successive events, see Section \ref{section: vehicle routing problems}.

For simplicity, we assume in this section  
that the probability distribution $\mathbb P_z := P \circ Z^{-1}$ of $Z$ 
has support contained in a \emph{compact convex} set $\mathsf Q$ with nonempty interior.
We also generally make the following assumption.
\begin{assumption}	\label{assumption: hyperplanes}
Hyperplanes in $\mathbb R^q$ have $\mathbb P_z$-measure zero.
\end{assumption}
Note that Assumption \ref{assumption: hyperplanes} implies that points also have measure zero, 
and in particular the support of $\mathbb P_z$ is infinite.
The following result, whose proof can be found in Appendix \ref{appendix: Lloyd gradient flow}, 
provides an expression for the derivatives of $\mathcal E_n$, useful for optimization purposes.
Throughout the paper $\partial \mathcal E_n / \partial p_i$ for $p_i \in \mathbb R^q$ denotes the $q$-dimensional 
vector of partial derivatives with respect to the components of $p_i$. 
We also adopt the convention $0/ \|0\| := 0$.

\begin{prop}	\label{prop: derivative of distortion}
Under Assumption \ref{assumption: hyperplanes},
$\mathcal E_n$ is Lipschitz continuous on compact sets  
and continuously differentiable on $\left( \mathbb R^q \right)^n \setminus \mathsf D_n$, 
with partial derivatives
\begin{equation}	\label{eq: partial derivative distortion}
\frac{\partial \mathcal E_n}{\partial p_i} \Big |_p = \int_{V_i(p)}  f'(\|p_i-z\|) \frac{p_i - z}{\|p_i - z\|} \mathbb P_z(dz).
\end{equation}
\end{prop}

Now let us suppose, as in \cite{cortesTRO04_coverage} and most of the subsequent literature, 
that the event location distribution $\mathbb P_z$ is known. 
Using (\ref{eq: partial derivative distortion}), one can then implement a gradient descent algorithm to 
locally minimize the objective (\ref{eq: coverage simple general}) \cite{cortesTRO04_coverage}. 
Assuming for simplicity that the agents are synchronized, and a constant sampling period $T>0$, 
we denote the agents positions at time $k T$ by $p_{k} := p(kT) = [p_{1,k},\ldots,p_{n,k}]$.
The robots start at $p_0=[p_{1,0},\ldots,p_{n,0}]$ at $t = 0$ and update their positions according to
\begin{equation}	\label{eq: standard coverage gradient update}
p_{i,k+1} = p_{i,k} - \gamma_k \frac{\partial \mathcal E_n}{\partial p_i} \Big|_{p_k},
\end{equation}
where $\gamma_k$ is an appropriately chosen sequence of decreasing or small constant positive stepsizes.
We ignore for the moment the issue of non-differentiability on $\mathsf D_n$ as well as the minor  
modifications required to accommodate 
velocity constraints in (\ref{eq: standard coverage gradient update}).
The agents can implement (\ref{eq: standard coverage gradient update}) to
asymptotically reach a configuration that is a critical point of $\mathcal E_n$. 
No guarantee to reach a global minimum is offered in general, and indeed global minimization of the 
function (\ref{eq: coverage simple general}) can be difficult \cite{fekete05_median}.
Nevertheless, an interesting property of the gradient descent algorithm (\ref{eq: standard coverage gradient update}) 
for the coverage control problem is that it can be implemented in a distributed manner by the robots, by exploiting the 
form of the expression (\ref{eq: partial derivative distortion}).
In particular, each agent can update its position at each period according to 
(\ref{eq: standard coverage gradient update}) by communicating only 
with its current Voronoi neighbors, in order to determine the boundaries of its own Voronoi cell $V_i(p)$ 
and compute the integral (\ref{eq: partial derivative distortion}). 
Even in a large network, a single robot has typically only few Voronoi neighbors \cite{Aurenhammer91_VoronoiSurvey}, 
which allows for a scalable and distributed implementation of the gradient descent algorithm.

\begin{remark}
The specific case where $f(x) = x^2$ is considered in \cite{cortesTRO04_coverage} in more details. In this case (\ref{eq: partial derivative distortion}) gives
\begin{equation}	\label{eq: gradient LS coverage}
\frac{\partial \mathcal E_n}{\partial p_i}|_{p = p_k} = 2 \mathbb P_z(V_{i}(p_k)) p_{i,k} - \int_{V_{i}(p_k)} z \mathbb P_z(dz) . 
\end{equation}
Assuming that $\mathbb P_z(V_{i}(p_k)) \neq 0$, define the centroid of the Voronoi region $V_{i}(p_k)$ as 
\[
C_{V_{i}(p_k)} = \frac{1}{\mathbb P_z(V_{i}(p_k))} \int_{V_{i}(p_k)} z \mathbb P_z(dz).  
\]
Then control law (\ref{eq: standard coverage gradient update}), i.e.,
\begin{align}	
p_{i,k+1} &= p_{i,k} - \gamma_k \frac{\partial \mathcal E_n}{\partial p_i} \Big|_{p_k} \nonumber \\
&= p_{i,k} - 2 \gamma_k \mathbb P_z(V_{i}(p_k)) (p_{i,k} -  C_{V_{i}(p_k)}), \label{eq: quadratic deterministic coverage control}
\end{align}
is essentially the well-known least-squares quantization algorithm of Lloyd \cite{Lloyd82_Quantization}.
\end{remark}

Note that the computation of the updates (\ref{eq: standard coverage gradient update}) requires $\mathbb P_z$ to be
perfectly known.
The minimization of (\ref{eq: coverage simple general}) is then essentially an open-loop optimization problem, and
the network can reach its desired configuration \emph{before any event occurs}. However, the algorithm 
does not provide any mechanism to adapt the configuration based on the actual observations of where the events occur,
which is critical in practice as $\mathbb P_z$ is rarely available. In the next section, we show how to generally address this issue 
by using stochastic gradient algorithms. 
Subsection \ref{subsection: stochastic coverage} applies the method specifically to the adaptive coverage control problem.


\subsection{Stochastic Gradient Algorithms}		\label{section: stochastic approximation theory}

Suppose that we wish to minimize a function $G$ defined on $\mathbb R^m$ for some $m \geq 0$, of the form
\begin{align}	
G(x) = E[g(x,Z)] &= \int_\Omega g(x,Z(\omega)) P(d \omega) \nonumber \\
&= \int_\mathsf Z g(x,z) \mathbb P_z(dz), \label{eq: expectation minimization}
\end{align}
such as $\mathcal E_n$ defined in (\ref{eq: coverage simple general}) for example.
The space $\mathsf Z$ in (\ref{eq: expectation minimization}) is the range of the random variable $Z$.
In contrast to the previous subsection, we now assume that $\mathbb P_z$ is unknown, 
so that the expectation (\ref{eq: expectation minimization}) cannot be computed directly. 
Suppose that $g$ is differentiable with respect to $x$, for $\mathbb P_z$-almost all $z$, 
and denote its gradient $\nabla_x g(x,z) := \frac{\partial g(x,z)}{\partial x}$. 
Finally, assume that we can observe random variables $Z_k, k \geq 1$, iid with distribution $\mathbb P_z$. 
Consider then the stochastic recursive algorithm
\begin{equation}	\label{eq: iterations EM algo}
x_{k+1} = x_k - \gamma_k \nabla_x g(x_k,Z_{k+1}),
\end{equation}
which can be rewritten in the form
\begin{equation}	\label{eq: SA general form}
x_{k+1} = x_k + \gamma_k (h(x_k) + D_{k+1}),
\end{equation}
with $h(x) := - E[\nabla_x g(x,Z_{1})|x]$ and $D_{k+1} = -\nabla_x g(x_k,Z_{k+1}) + E[\nabla_x g(x_k,Z_{k+1})|x_k]$.
Define for $k \geq 0$ the filtration $\mathcal F_k := \sigma(x_0, D_i, 1 \leq i \leq k)$,
i.e., an increasing family $\mathcal F_k \subset \mathcal F_l$ for $k \leq l$ of sub-$\sigma$-algebras of $\mathcal F$.
Then $\{D_k\}_{k \geq 1}$ is a \emph{martingale difference sequence} (MDS) with respect to $\{\mathcal F_k\}_{k \geq 0}$, as explained in the following definition.
\begin{definition}
Let $\{\mathcal F_k\}_{k \geq 0}$ be a filtration. 
A sequence of random variables $\{ D_k \}_{k \geq 1}$ is called a martingale difference sequence
with respect to $\{\mathcal F_k\}_{k \geq 0}$ if $D_{k}$ is measurable with respect to $\mathcal F_{k}$, 
$E[\|D_{k}\|] < \infty$, and we have
$E[D_{k} | \mathcal F_{k-1}] = 0$, for all $k \geq 1$.
\end{definition}

Intuitively, the MDS $\{D_k\}_{k \geq 1}$ plays the role of a zero-mean noise. 
By the ODE method \cite{Ljung77_SA}, we can expect that asymptotically, under the condition
\begin{equation}	\label{eq: classical stepsize condition}
\gamma_k \geq 0, \;\; \sum_{k = 0}^\infty \gamma_k = +\infty, \;\; \sum_{k = 0}^\infty \gamma_k^2 < +\infty,
\end{equation}
on the stepsizes, which holds for $\gamma_k = 1/(1+k)$ for example, 
the sequence $\{x_k\}_{k \geq 0}$ in (\ref{eq: SA general form}) 
almost surely approaches the trajectories of the ODE
\begin{equation}	\label{eq: limit ODE}
\dot x = h(x).	
\end{equation}
Now assume that it is valid to exchange expectation and derivation, as follows
\begin{align}	
- \nabla G(x) &= -\nabla E[g(x,Z_1)|x] \nonumber \\
&= -E[\nabla_x g(x,Z_1) | x] = h(x). \label{eq: diff under integral}
\end{align}
Identity (\ref{eq: diff under integral}) can often be proved using the 
dominated convergence theorem, see, e.g., \cite[Theorem 5.1]{Spall03_SA}.
Then the ODE (\ref{eq: limit ODE}) describes a gradient flow and so in fact under mild assumptions
the trajectories and therefore the sequence $\{x_k\}_{k \geq 0}$ approach the critical points of $G$.
Moreover, we can often expect convergence to the set of local minima of $G$ almost surely \cite[chapter 4]{Borkar08_SA}. 
In conclusion, the algorithm (\ref{eq: SA general form}) allows us to reach such a minimum even though $\mathbb P_z$ is unknown,
as long as we have access to realizations of the random variables $Z_k$.

We now capture the intuition above more formally,
including the situation where the function $G$ is not everywhere differentiable, 
as in Proposition \ref{prop: derivative of distortion}.
Consider a stochastic algorithm 
\begin{align}	\label{eq: gradient SA general}
x_{k+1} = x_k + \gamma_k (h_k + D_{k+1}), \;\; \forall k \geq 0,
\end{align}
where the stepsizes $\gamma_k$ satisfy (\ref{eq: classical stepsize condition}), 
$\{D_k\}_{k \geq 1}$ is an MDS with respect to the filtration $\mathcal F_k := \{x_l,h_l,D_l,l \leq k\}, k \geq 0$, 
and $h_k$ is specified in the following theorems.

\begin{thm}	\label{thm: standard thm ODE}
Assume that $G$ is continuously differentiable on $\mathbb R^m \setminus \mathsf S$, 
with $\mathsf S$ a set of Lebesgue measure zero.  
Introduce the Filipov set-valued map \cite{Cortes08_discontinuous}
\begin{equation}	\label{eq: set value map gradient}
\mathcal H(x) 
= \begin{cases}
\{-\nabla G(x)\}, & x \notin \mathsf S, \\
\bigcap_{\delta > 0} \overline{\text{co}} \left( \bigcup_{\hat x \in B(x,\delta) \setminus \mathsf S} \{- \nabla G(\hat x)\} \right),
& x \in \mathsf S,
\end{cases}
\end{equation}
where $\overline{\text{co}}$ denotes the closed convex hull.
Consider the recurrence (\ref{eq: gradient SA general}) 
with $h_k \in \mathcal H(x_k)$, for all $k \geq 0$.
Assume that for some positive constants $K_1, K_2$ we have
\begin{align*}
& \sup_{h \in \mathcal H(x)} \|h\| \leq K_1 (1+\|x\|), \;\; \forall x \in \mathbb R^m, \\
& E[\|D_{k+1}\|^2 | \mathcal F_k] \leq K (1+\|x_k\|^2), \;  \text{ a.s. }, \forall k \geq 0, 
\end{align*}
and that $\sup_{k \geq 0} \|x_k\| < \infty,$ a.s.
Then the sequence $\{x_k\}_{k \geq 0}$ converges almost surely 
to a connected subset of $\{ x \in \mathbb R^m \setminus \mathsf S | \nabla G(x) = 0 \} \cup \mathsf S$, 
invariant for the differential inclusion $\dot x \in \mathcal H(x)$.
\end{thm}

\begin{thm}	\label{thm: standard thm ODE - convex case}
Assume that $G$ is convex and admits a minimum on $\mathbb R^m$.
Consider the recurrence (\ref{eq: gradient SA general}) with 
$h_k$ a subgradient of $G$ at $x_k$, for all $k \geq 0$.
Assume that there exists a positive constant $K$ such that
$E[\|h_k + D_{k+1}\|^2| \mathcal F_k] \leq K$, for all $k \geq 0$.
Then the sequence $\{x_k\}_{k \geq 0}$ converges almost surely to some point minimizing $G$.
\end{thm}

The proofs of these theorems are standard and not repeated here, 
see \cite{Kushner03_SA}, \cite[chapter 5]{Borkar08_SA}, \cite[Proposition 8.2.6. p. 480]{Bertsekas03_convexBook}
and the proof of Theorem \ref{thm: main result coverage} in Appendix \ref{section: appdx pf coverage control}.
Note that in many applications, the stepsizes $\gamma_k$ are chosen to converge to a small positive constant instead of
satisfying (\ref{eq: classical stepsize condition}), which allows tracking of the equilibria of the gradient flow
if the problem parameters (e.g., $\mathbb P_z$) change with time.
In this case, one typically obtains convergence to a neighborhood of a critical point \cite{Benveniste90_SA}.
The selection of proper stepsizes is an important practical issue that is not emphasized in this paper 
but is discussed at length in references on stochastic approximation algorithms \cite{Kushner03_SA, Spall03_SA}.


\subsection{Adaptive Coverage Control} \label{subsection: stochastic coverage}

We now consider the following modification of the coverage control problem. 
The events occur randomly in the workspace, with event $k \geq 1$ occuring at time $t_k > 0$ and 
location $Z_k \in \mathsf Q$. We let $t_0 := 0$ denote the initial time. 
Assume in this subsection that the successive locations of the events $Z_k, k \geq 1,$ 
are iid with probability distribution $\mathbb P_z$ on $\mathsf Q$.
The distribution $\mathbb P_z$ is now unknown, and as a result the deterministic gradient descent 
algorithm (\ref{eq: standard coverage gradient update}) cannot be implemented. 
We work under Assumption \ref{assumption: hyperplanes}, 
so that the gradient expression (\ref{eq: partial derivative distortion}) holds.

We denote the agent positions at time $t^-_k$, i.e., right before the occurrence of 
the $k^{th}$ event, by $p_{k-1} = [p_{1,{k-1}},\ldots,p_{n,{k-1}}] \in (\mathbb R^{q})^n$, for $k \geq 1$. 
These positions are called \emph{reference positions} and are updated according to
\begin{equation}	\label{eq: trivial dynamics discrete time}
p_{i,k+1} = p_{i,k} + u_{i,k}, \;\; |u_{i,k}| \leq v_{i,k}, \; \forall k \in \mathbb Z_{\geq 0}, \forall i \in [n],
\end{equation}
where $u_{i,k} \in \mathbb R^q$ is a control input for the interval $[t_k,t_{t+1})$. 
For example, if the robot dynamics follow the model (\ref{eq: trivial dynamics continuous time}) 
and if servicing the targets requires no additional travel, we can take $v_{i,k} = v_i (t_{k+1} - t_k)$ for all $i \in [n]$. 
We assume that there exists a constant $v>0$ such that $v_{i,k} \geq v$ for all $i \in [n]$ and $k \geq 0$, 
so that the robots can update their reference positions by a non-vanishing positive distance at each period.

When the $k^{th}$ event occurs at time $t_k$ and position $Z_k \in \mathsf Q$, 
we assume that at least the robot closest to that event location can observe it. 
This robot, say robot $i$, services the target starting from its location $p_{i,k-1}$, 
and then moves to a new reference position $p_{i,k}$.
The following reference position updates 
implement the stochastic gradient algorithm (\ref{eq: iterations EM algo}) 
to minimize the coverage objective (\ref{eq: coverage simple general}).
First, for a vector $u \in \mathbb R^q$ and a scalar $b > 0$, define the truncation $[\text{sat}(u)]_{b}$ by
\begin{equation*}
[\text{sat}(u)]_{b} = \begin{cases}
u, & \text{ if } \|u\| \leq b, \\
b \frac{u}{\|u\|}, & \text{ if } \|u\| > b.
\end{cases}
\end{equation*}
Then consider the update rule
\begin{align}	
&p_{i,k+1} = \label{eq: update law Voronoi case} \\
&\begin{cases}
\Pi_\mathsf Q \left [p_{i,k} + \text{sat}\left[ \gamma_k f'(\|p_{i,k} - Z_{k+1}\|) 
\frac{Z_{k+1}-p_{i,k}}{\|Z_{k+1}-p_{i,k}\|} \right]_{v_{i,k}} \right] \\
\hspace{2.5cm} \text{if } i \in \arg \min_{j \in [n]} \|p_{j,k}-Z_{k+1}\|, \\
p_{i,k} \text{ otherwise,} \nonumber
\end{cases}
\end{align}
where $\Pi_\mathsf Q$ is the orthogonal projection on the convex set $\mathsf Q$. 
Note that the situation where several robots are at equal distance from $Z_k$ 
and simultaneously update their reference position 
occurs with probability zero under Assumption \ref{assumption: hyperplanes}.
To justify (\ref{eq: update law Voronoi case}) based on
the discussion in the previous subsection, let 
$
g(p,z) = \min_{i \in [n]} f(\|p_i-z\|)$, i.e., $g(p,z) = f(\|p_{i^*}-z\|) \text{ for } z \in V_{i^*}(p)$.
Then we have
\[
\frac{\partial g}{\partial p_i}(p,z) = \begin{cases}
f'(\|p_i-z\|) \frac{p_i - z}{\|p_i - z\|}, & \text{if } z \in \text{Int}(V_i(p)) \setminus \{p_i\} \\
0, & \text{if } z \notin V_i(p).
\end{cases}
\]
Moreover let us define $(\partial g / \partial p)(p,z)$ arbitrarily for $z$ on the 
Voronoi cell boundaries and at the points $p_i$. These sets have $\mathbb P_z$-measure zero 
under Assumption \ref{assumption: hyperplanes}, and hence do not contribute to the integral
\begin{align*}
E \left[ \frac{\partial g}{\partial p_i}(p,Z) \right] 
&= \int_{V_i(p)} f'(\|p_i-z\|) \frac{p_i - z}{\|p_i - z\|} \mathbb P_z(dz) \\
&= \frac{\partial \mathcal E_n}{\partial p_i}(p), \;\; \text{for $p \notin \mathsf D_n$}.
\end{align*}
In other words, Proposition \ref{prop: derivative of distortion} precisely says that 
the identity (\ref{eq: diff under integral})
is valid for $\mathcal E_n$ in $(\mathbb R^q)^n \setminus \mathsf D_n$.
Note also that almost surely the update rule (\ref{eq: update law Voronoi case}) 
never results in two robots landing on the same position as long as $q \geq 2$ and the
updated reference position before projection remains in $\mathsf Q$, 
because this would require $Z_{k+1}$ to fall on 
the line passing through these two robot reference positions.
This can be achieved for $q=1$ or for a reference position projected 
on the boundary of $Q$ as well, by a small random perturbation of the sequence $\gamma_k$ \cite[Chapter 2]{Borkar08_SA}. 
Hence we can assume in the following that almost surely $p_k \notin \mathsf D_n$ for all $k \geq 1$.
Moreover, the projection $\Pi_\mathsf Q$ and the saturation nonlinearity
do not change the convergence properties of the algorithm, see Appendix \ref{appendix: Lloyd gradient flow}.
Therefore, (\ref{eq: update law Voronoi case}) is essentially the stochastic gradient descent 
update rule (\ref{eq: iterations EM algo}).

It is interesting to compare the implementation complexity of algorithm (\ref{eq: update law Voronoi case}) 
with that of the corresponding deterministic gradient descent update based 
on (\ref{eq: partial derivative distortion}),  (\ref{eq: standard coverage gradient update}).
The deterministic, model-based algorithm requires that each agent 
maintains communication with its Voronoi neighbors and knows their position
in order to determine the boundaries of its Voronoi cell and compute 
the integral (\ref{eq: partial derivative distortion}). 
Even in the quadratic case (\ref{eq: quadratic deterministic coverage control}), 
this scheme can be difficult to implement.
In contrast, no Voronoi cell computation or integration and no detailed knowledge 
of the position of the neighbors is required by (\ref{eq: update law Voronoi case}), 
which only needs a distributed mechanism to find which robot is the closest to the target when it appears.
This can be done in a distributed way via the \texttt{FloodMin} algorithm described in Paragraph \ref{section: min-consensus}, 
with the agents initializing their value to $\hat d_i = \|p_{i,k} - Z_{k+1}\|$ if they detect the event, and to $\hat d_i = + \infty$ 
if they are too far away to detect it. 
Clearly there are other ways to implement the rule (\ref{eq: update law Voronoi case}). 
For example, we could let all the robots travel to the event location at the same speed, 
as in \cite{Arsie09_routing}, a scheme that does not require any coordination.  
Then only the first robot to reach the target changes its reference position for the next period.

\emph{Special Cases:} If we specialize (\ref{eq: update law Voronoi case}) 
to the least-squares coverage control problem with $f(x) = x^2$ and ignore the saturation function,
we obtain the update $p_{i,k+1} = p_{i,k}+\gamma_k (Z_{k+1}-p_{i,k})$ for the closest robot. 
This particular adaptive algorithm has been used extensively in various fields, 
from statistics to quantization to neural 
networks \cite{MacQueen67_Kmeans, Kohonen82_SOM, Gray84_vectorQuantization}.
If $f(x) = x$ and all robots travel at unit speed, the service cost for 
an event appearing at $Z_k$ is the time it takes for the closest robot 
to travel to the event location. In this case, the update rule (\ref{eq: update law Voronoi case}) 
is simply $p_{i,k+1} = p_{i,k}+\gamma_k \frac{Z_{k+1}-p_{i,k}}{\|Z_{k+1}-p_{i,k}\|}$ 
for the closest robot. It is used in the vehicle routing application discussed
in Section \ref{section: vehicle routing problems}.

\begin{remark}
For certain distributions and initial robot positions outside of the support set of the distribution $\mathbb P_z$, 
it is possible that by following (\ref{eq: update law Voronoi case}), some agents never move. 
The issue also arises with the deterministic algorithm, 
since the gradient (\ref{eq: gradient LS coverage}) vanishes if $\mathbb P_z(V_{i}(p_k)) = 0$.
A possible solution to avoid this phenomenon is to add an initial transient regime 
where for example all agents follow the first case of the rule (\ref{eq: update law Voronoi case}) 
rather than only the closest agent. The goal of this transient modification is to bring 
all the robots within the support set of the event distribution. 
It is either stopped at some finite time or discounted by a stepsize decreasing 
much faster that $\gamma_k$, thereby not impacting the convergence results \cite{Borkar08_SA}.
\end{remark}

We now state a convergence result for the update law (\ref{eq: update law Voronoi case}) 
to the set of critical points of the objective $\mathcal E_n$, i.e., to
\begin{equation}	\label{eq: critical points of distortion}
\mathsf C_n = \{p \in \mathsf Q^n \setminus \mathsf D_n | \nabla \mathcal E_n(p) = 0\}.
\end{equation}
Even though the algorithm is a stochastic gradient algorithm, the discontinuity of $\nabla \mathcal E_n$ 
on the set $\mathsf D_n$ creates technical difficulties. To the best of our knowledge, the most thorough 
investigation of the dynamics of (\ref{eq: update law Voronoi case}) can be found in \cite{Pages98_quantization} 
and leaves open the question of non-convergence to $\mathsf D_n$.
In contrast to that paper, we cope with the non-differentiability on $\mathsf D_n$ 
by introducing the Filippov set-valued map $\mathcal H_n$ as in (\ref{eq: set value map gradient})
\begin{equation}	\label{eq: set value map gradient - specific}
\mathcal H_n(x) 
= \begin{cases}
\{-\nabla \mathcal E_n(x)\}, & x \notin \mathsf D_n, \\
\bigcap_{\delta > 0} \overline{\text{co}} \left( \bigcup_{\hat x \in B(x,\delta) \setminus \mathsf D_n} \{- \nabla \mathcal E_n(\hat x)\} \right),
& x \in \mathsf D_n.
\end{cases}
\end{equation}
We also need the following definition. A Borel measure $\mu$ on $\mathbb R^q$ is said to dominate the 
Lebesgue measure $\lambda$ if $\lambda(A)  = 0$ for all Borel sets $A$ such that $\mu(A) = 0$. 

\begin{thm}	\label{thm: main result coverage}
Let the stepsizes $\gamma_k$ satisfy (\ref{eq: classical stepsize condition}), $p_0 \in \mathsf Q^n \setminus \mathsf D_n$, and 
suppose that Assumption \ref{assumption: hyperplanes} holds.
Then, by following the algorithm (\ref{eq: update law Voronoi case}), 
the sequence $\{p_k\}_{k \geq 0}$ of robot positions converges almost surely to 
a compact connected subset of $\mathsf C_n \cup (\mathsf D_n \cap \mathsf Q^n)$, 
invariant for the differential inclusion $\dot p \in \mathcal H_n(p)$.

If in addition $\mathbb P_z$ dominates the Lebesgue measure on $\mathsf Q$, 
then the sequence $\{p_k\}_{k \geq 0}$ converges almost surely to a compact connected subset 
of $\mathsf C_n$. In particular if $\mathcal E_n$ has only isolated critical points 
in $\mathsf Q^n \setminus \mathsf D_n$, the sequence $\{p_k\}_{k \geq 0}$ 
converges to one of them almost surely. 
\end{thm}

The proof of Theorem \ref{thm: main result coverage} can be found in appendix \ref{appendix: Lloyd gradient flow}. 
The first part of the theorem is a fairly direct application of Theorem \ref{thm: standard thm ODE}, 
but does not rule out asymptotic convergence to the set $\mathsf D_n$ of aggregated configurations.
This motivates the second part of the theorem.


\subsection{Some Extensions}

Before closing this section, we briefly illustrate how the stochastic gradient view
leads to simple solutions for interesting variations of the coverage control problem.

\subsubsection{A Heterogenous Coverage Problem}	\label{section: heterogeneous coverage}

As in Subsection \ref{subsection: stochastic coverage}, 
an event appears randomly in the environment at each period and must be serviced. 
However, let us now assume that there are two types of agents, 
with $m_A$ robots of type $A$ and $m_B$ robots of type $B$, and three types of events: $a,b$, and $ab$.
 Events of type $a$ must be serviced by a robot of type $A$, events of type $B$ by a robot of type $b$, 
 and events of type $ab$ by a robot of type $A$ \emph{and} a robot of type $B$. 
When a new event appears, it is of type $\alpha$ with some unknown probability $\lambda_\alpha$, 
$\alpha \in \{a,b,ab\}$, and the agents can observe its type.
The spatial distribution $\mathbb P_\alpha$ of events of type $\alpha$ 
is also a priori unknown, and satisfies Assumption {\ref{assumption: hyperplanes}.
Finally, denote the vector of robot positions $p=[p_1^A,\ldots,p_{m_A}^A, p_1^B,\ldots,p_{m_B}^B]$.
The asymptotic configuration of the robots must now optimize the expected cost
\begin{align} \label{eq: objective heterogeneous} 
& \mathcal E_{m_A,m_B}(p) = \lambda_a E \left[ \min_{i \in [m_A]} f_A(\|p^A_i - Z\|) \Big | \alpha = a \right  ] \\
& + \lambda_b E \left[ \min_{j \in [m_B]} f_B(\|p^B_j - Z\|) \Big | \alpha = b \right ]  
 + \lambda_{ab} E \bigg[ \;  \max \bigg \{ \nonumber \\
& \min_{i \in [m_A]} f_A(\|p^A_i -Z\|) , 
 \min_{j \in [m_B]} f_B(\|p^B_j -Z\|) \bigg \} \bigg | \alpha = ab \bigg], \nonumber
\end{align}
where $f_A$ and $f_B$ are increasing, continuously differentiable functions with values in $\mathbb R_{\geq 0}$. 
Note that the cost of servicing an event of type $ab$ is the maximum of the 
costs of servicing it with one robot of each type. 
This can model the time necessary for one robot of each type to travel to the event location for example.

For this problem, one can verify as before that the stochastic gradient update rule (\ref{eq: iterations EM algo})
takes the following surprisingly simple form.  
When an event of type $a$ appears at location $z_{k+1}$, the closest robot of type $A$, say $i$, 
services it and changes it reference position by moving it toward $z_{k+1}$ 
by a (truncated and projected) step $\gamma_k f_A'(\|z_{k+1}-p^A_{i,k}\|)\frac{z_{k+1}-p^A_{i,k}}{\|z_{k+1}-p^A_{i,k}\|}$
as in (\ref{eq: update law Voronoi case}), 
and similarly for a target of type $b$ and a robot of type $B$. If the target is of type $ab$, 
the closest $A$ and $B$ robots service it. 
To update their reference positions for the next period, they first find 
which of the two is the \emph{farthest} from the event location. 
Then only this robot moves its reference position by the same step as in (\ref{eq: update law Voronoi case}). 
In view of the complicated expression of the objective function, 
such a simple rule based update law is quite appealing. 
We illustrate its behavior on Fig. \ref{fig: heterogeneous vehicle routing} 
for $f_A(x)=f_B(x)=x$.

\begin{figure}
\centering
\includegraphics[width=\linewidth]{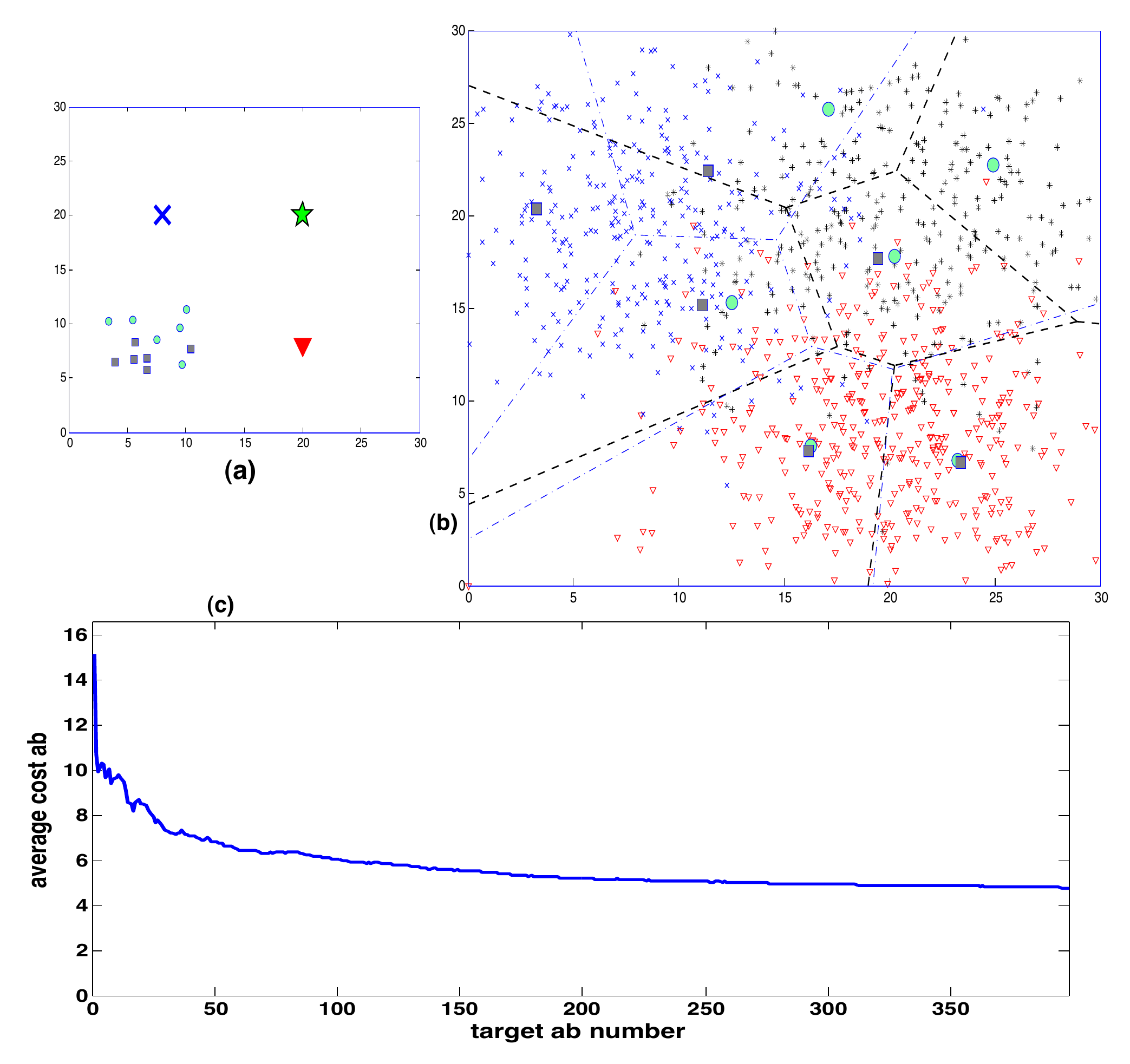}
\caption{Heterogeneous coverage control for a system with two types of robots, 
$A$ (green circles) and $B$ (gray squares).  Events requiring service from 
type $a$ appear with probability $30\%$ and a distribution approximately 
centered at $[20;20]^T$ (star on Fig. \ref{fig: heterogeneous vehicle routing}(a)). 
Targets of type $b$ appear with probability $30\%$ and a distribution approximately 
centered at $[8;20]^T$ (cross on Fig. \ref{fig: heterogeneous vehicle routing}(a)). 
Finally targets of type $ab$ appear with probability $40\%$ and a distribution 
approximately centered at $[20;8]^T$ (triangle on Fig. \ref{fig: heterogeneous vehicle routing}(a)). 
Fig. \ref{fig: heterogeneous vehicle routing}(a) shows the 
initial robot configuration and Fig. \ref{fig: heterogeneous vehicle routing}(b) the configuration reached after $1000$ targets, 
together with the history of target locations. The Voronoi cells of each robot are 
indicated but not computed by the algorithm (separate Voronoi diagrams are drawn 
for the two robot types). Note how robots of type $A$ and $B$ tend to pair in the lower right 
corner in order to service the targets of type $ab$ efficiently (here $f_A(x)=f_B(x)=x$). 
Fig. \ref{fig: heterogeneous vehicle routing}(c) shows the empirical average cost incurred by the targets of type $ab$, 
where the average is taken over all the past targets of this type seen so far.}
\label{fig: heterogeneous vehicle routing}
\end{figure}

\subsubsection{Target Tracking with Markovian Dynamics}

Suppose now that we wish to track a single target in discrete time, 
whose position at time $t_k$ is $Z_k$, where $Z_k$ evolves as 
a Markov chain with a unique stationary asymptotic distribution $\mathbb P_z$.
The  objective is still to optimize $\mathcal E_n$ defined by (\ref{eq: coverage simple general}), 
which represents the steady-state tracking error. 
We can then use algorithm (\ref{eq: update law Voronoi case}) 
to optimize the steady-state robotic network configuration, 
and a convergence result similar to Theorem \ref{thm: main result coverage} can be proven
using stochastic approximation arguments \cite[Chapter 1]{Benveniste90_SA}.
This tracking scheme does not require the knowledge of the target dynamics 
nor that of the stationary distribution $\mathbb P_z$.

As an example, consider a target moving on a circle of radius $R$, with dynamics
\[
\theta_{k+1} = 0.95 \, \theta_k + \xi_k,
\]
where the variables $\xi_k$ are iid uniform on $[-0.5,0.5]$ and $Z_k = [R \cos \theta_k, R \sin \theta_k]^T$. 
The result of the adaptive coverage algorithm for $f(x) = x^2$ is 
shown on Fig. \ref{fig: Markov dynamics on a circle}. 
Although the target distribution does not dominate the Lebesgue measure 
as required in the second part of Theorem \ref{thm: main result coverage}, 
in practice we do not observe convergence to an aggregated configuration. 
Note how the robots position themselves in the region 
around the point $[1, 0]^T$ where the target spends most of its time.

\begin{figure}[ht]
\centering
\includegraphics[width = \linewidth]{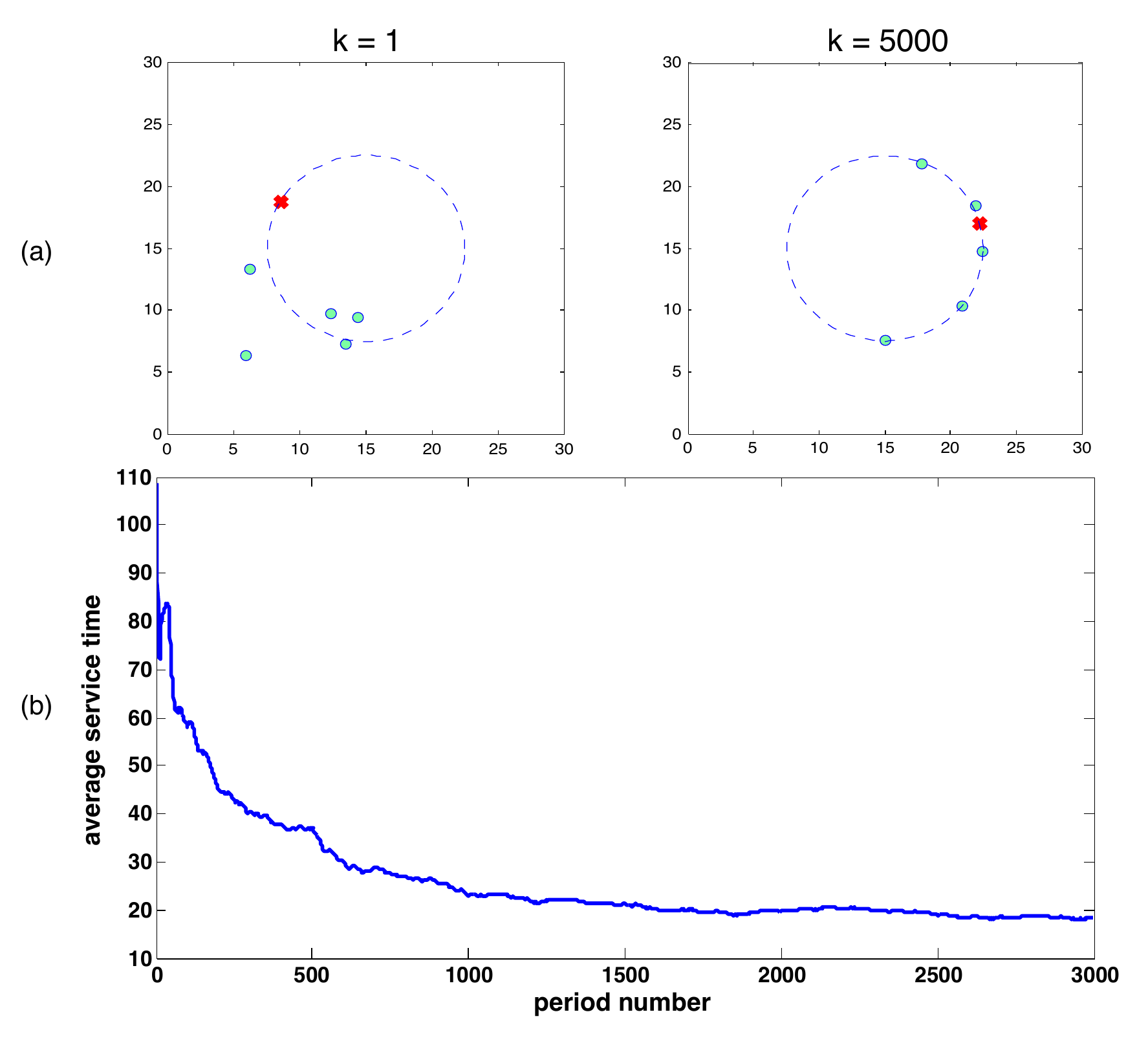}
\caption{Adaptive coverage algorithm for a target with Markovian dynamics moving on a circle. 
We show on Fig. \ref{fig: Markov dynamics on a circle}(a) the positions of the robots (blue circles) and the target (red cross) initially and after $5000$ time-steps. 
The stepsizes used were $\gamma_k = 1/(1+5 \times 10^{-3} k)$. 
The curve on Fig. \ref{fig: Markov dynamics on a circle}(b) shows the evolution of the empirical average cost over time, 
where the average is taken over the past $1000$ cost measurements.}
\label{fig: Markov dynamics on a circle}
\end{figure}


\section{Adaptive Spatial Load-Balancing and Partitioning}		
\label{section: load-balancing}

In this section, we design distributed adaptive algorithms that 
partition the space into $n$ cells, one for each robot, 
so that the steady-state probability that an event falls into cell $i$ has a prespecified value $a_i$. 
Here we have $a_i > 0$ for all $i \in [n]$, and $\sum_{i=1}^n a_i = 1$. 
These algorithms allow an operator to specify the steady state utilization 
of the different robots, by letting each robot service only the events occurring in its cell.
Such spatial load balancing algorithms have important applications in multi-robot systems 
and location optimization, 
see e.g. \cite{Aurenhammer98_clustering, Pavone09_equipartitions, Cortes10_loadBalancing}. 
An application to the DTRP is described in Section \ref{section: vehicle routing problems}.

As in Section \ref{subsection: stochastic coverage}, events occur at times $t_k$ 
and iid locations $Z_k, k \geq 1$, and the unknown distribution $\mathbb P_z$ 
has support included in $\mathsf Q$. In this section, $\mathsf Q$ is assumed to be
 compact for simplicity, but not necessarily convex. 
Based on the observation of the successive event locations, 
we design a sequence of partitions of $\mathsf Q$ into regions $\{R_{i,k}\}_{i \in [n]}$, $k \geq 0$, 
such that at period $k \geq 1$, agent $i$ is responsible for servicing the event if and only if $Z_k \in R_{i,k-1}$. 
Here we slightly abuse terminology and allow our partitions to have $R_{i,k} \cap R_{j,k} \neq \emptyset$ 
for $i \neq j$. 
Our algorithms produce regions whose intersections have $\mathbb P_z$-measure zero, 
hence this has no influence on the final result. 
After the $k^{th}$ event occurs, the agents can change the boundaries of their respective regions 
to form the partition $\{R_{i,k}\}_{i \in [n]}$ used to decide which agent services the $(k+1)^{th}$ event. 
Our sequence of partitions $\{R_{i,k}\}_{i \in [n]}$ converges 
to a partition $\{R_{i}\}_{i \in [n]}$, 
i.e., $d_H(R_{i,k},R_i) \to 0$ as $k \to \infty$, 
such that $\mathbb P_z(R_i) = a_i$ for all $i \in [n]$. 

Let $\mathcal G = \{g_1,\ldots,g_n\}$ be a set of $n$ fixed and distinct points in $\mathbb R^q$, 
with point $g_i$ associated to robot $i$. 
We call the point $g_i$ the generator of region $R_i$. 
Designing a partition $\{R_i\}_{i \in [n]}$ is equivalent to choosing 
an \emph{assignment} of event locations to region generators, 
i.e., a measurable map $T: \mathsf Q \to \mathcal G$, by taking $R_i = T^{-1}(g_i), i \in [n]$. 
Let us denote the set of all such assignments by $\mathcal T$. 
We then look for an assignment $T \in \mathcal T$ satisfying 
the constraint $\mathbb P_z(T^{-1}(g_i)) = a_i, i \in [n]$, 
and design recursive algorithms producing such an assignment asymptotically.
Now consider the following optimization problem 
\begin{align}	\label{eq: Monge problem}
\inf_{T \in \mathcal T} \quad & \int_\mathsf Q c(z,T(z)) \mathbb P_z(dz) \\
\text{subject to } \quad & \mathbb P_z(T^{-1}(g_i)) = a_i, \;\; i \in [n],		\label{eq: partitioning constraint}
\end{align}
where $c: \mathsf Q \times \mathcal G \to \mathbb R$ is a given cost function. 
The following theorem gives a general way of producing assignments or partitions 
that optimize (\ref{eq: Monge problem}), (\ref{eq: partitioning constraint}).

\begin{thm}	\label{thm: partitioning duality theorem}
Consider problem (\ref{eq: Monge problem}), (\ref{eq: partitioning constraint}), 
where $\mathsf Q$ is compact, and assume that
\begin{enumerate}[{A}1)]
\item \label{eq: assumption on the cost}For all $i \in [n]$, $z \to c(z,g_i)$ 
is lower bounded and lower semi-continuous on $\mathsf Q$, 
and $z \to \max_{i \in [n]} c(z,g_i)$ is $\mathbb P_z$-integrable.
\item \label{eq: regularity of P} For all $i \neq j \in [n]$, 
for all $r \in \mathbb R$, the set $\{z \in \mathsf Q: c(z,g_i)-c(z,g_j) = r\}$ 
has $\mathbb P_z$-measure zero.	\label{assumption: regularity of Pz}
\end{enumerate}
Then the problem admits an assignment $T \in \mathcal T$ 
that attains the infimum in (\ref{eq: Monge problem}). 
The value of the optimization problem is equal to
\begin{equation}	\label{eq: dual function maximization}
\max_{w \in \mathbb R^n} h(w) := \int_{\mathsf Q} \min_{i \in [n]} \{ c(z,g_i) - w_i \} \; \mathbb P_z (dz) 
+ \sum_{i=1}^n a_i w_i,
\end{equation}
and this maximum is attained for some $w^* \in \mathbb R^n$. 
An optimal assignment $T$ is then given by the generalized Voronoi regions
\[
\forall z \in \mathsf Q, \;\; T(z) = g_i \Leftrightarrow z \in V^c_i(\mathcal G,w^*).
\]
The function $h$ is concave, and a supergradient of $h$ at $w$ is given by 
\begin{equation}	\label{eq: supergradient for generalized Voronoi}
[- \mathbb P_z(V^c_1(\mathcal G,w)) + a_1, \ldots, - \mathbb P_z(V^c_n(\mathcal G,w)) + a_n]^T.
\end{equation}
Finally, the following supergradient optimization algorithm
\begin{align}
& w_0 = 0, \nonumber \\
& w_{i, k+1} = w_{i, k} + \gamma_k [- \mathbb P_z(V^c_i(\mathcal G,w_k)) + a_i], 
\; i=1,\ldots,N, \label{eq: standard gradient descent}
\end{align}
where $\gamma_k$ is a sequence of stepsizes satisfying (\ref{eq: classical stepsize condition}), 
converges to an optimal set of weights maximizing $h$.
\end{thm}

In other words, there is a set of weights $w^* \in \mathbb R^n$, 
maximizing the dual function defined in (\ref{eq: dual function maximization}), 
for which the corresponding generalized Voronoi cells $\{V^c_i(\mathcal G,w^*)\}_{i \in [n]}$ 
defined in (\ref{eq: generalized Voronoi})
satisfy the constraints of interest (\ref{eq: partitioning constraint}). 
In addition, the assignment corresponding to these regions minimizes 
the expected cost (\ref{eq: Monge problem}). 
In practice, we make additional assumptions on the function $c$ 
to obtain reasonably shaped regions. In particular, if $c(z,g_i) = \|z-g_i\|^2$, 
then the generalized Voronoi diagram becomes a power diagram.
Because the boundaries of the power cells are hyperplanes in $\mathbb R^q$  \cite{Okabe2000_Voronoi}, 
our Assumption A\ref{eq: regularity of P} on $\mathbb P_z$ in Theorem \ref{thm: partitioning duality theorem} 
is satisfied under Assumption \ref{assumption: hyperplanes}.
Theorem \ref{thm: partitioning duality theorem} generalizes some results 
in \cite{Aurenhammer98_clustering, Pavone09_equipartitions, Cortes10_loadBalancing} 
by imposing weaker conditions on $\mathbb P_z$ and $c$.  
A proof is provided in Appendix \ref{appendix: optimal transport}, 
based on results from 
optimal transportation \cite{Rachev98_optimalTransport, Gangbo96_optimalTransport, Ruschendorf97_coupling}.

For our scenario where $\mathbb P_z$ is unknown, 
we replace the gradient ascent algorithm (\ref{eq: standard gradient descent})
by a stochastic version presented in Algorithm \ref{algo: adaptive partitioning}, 
and whose behavior is illustrated on Fig. \ref{fig: partitioning}.
For simplicity, we specialize the discussion to $c(z,g_i) = f(\|z-g_i\|)$, 
where $f$ is increasing, and denote the generalized Voronoi cells by $V^f_i(\mathcal G,w)$.
If, at period $k$, the event is located at $Z_k$, a possible choice 
for the stochastic supergradient is simply
\begin{equation}	\label{eq: stochastic gradient partitioning}
[-\mathbf 1_{\{V_1^f(\mathcal G,w_{k-1})\}}(Z_k) + a_1, \ldots, 
-\mathbf 1_{\{V_n^f(\mathcal G,w_{k-1})\}}(Z_k) + a_n]^T.
\end{equation}
Computing component $i$ of (\ref{eq: stochastic gradient partitioning}) relies on 
testing if $Z_k \in V_i^f(\mathcal G,,w_{k-1})$, which is much easier than 
computing the $\mathbb P_z$-area of the generalized Voronoi cell as 
in (\ref{eq: standard gradient descent}).
For this test, assuming that at least the robot associated with the region $R_{i,k-1}$ 
where the $k^{th}$ event occurs detects the event, the agents can simply 
run the \texttt{FloodMin} algorithm (see Subsection \ref{section: min-consensus}) 
with $\hat d_i = f(\|Z_k - g_i\|) - w_{i,k-1}$ (and $\hat d_i = +\infty$ if agent $i$ did not detect the event).
The following result is now a direct application of Theorem \ref{thm: standard thm ODE - convex case}.
\begin{thm}		\label{thm: load balancing convergence}
Assume that the stepsizes $\gamma_k$ in Algorithm \ref{algo: adaptive partitioning} 
satisfy (\ref{eq: classical stepsize condition}), and that Assumptions A\ref{eq: assumption on the cost}, 
A\ref{eq: regularity of P} of Theorem \ref{thm: partitioning duality theorem} 
are satisfied for $c(z,g_i) = f(\|z-g_i\|)$. 
Then the weights updated by following Algorithm \ref{algo: adaptive partitioning} 
converge almost surely to a maximizer $w^*$ of (\ref{eq: dual function maximization}), 
and the resulting generalized Voronoi diagram $\{V^c_i(\mathcal G,w^*)\}_{i \in [n]}$ 
satisfies the utilization constraints (\ref{eq: partitioning constraint}).
\end{thm}

\begin{algorithm}
\caption{Adaptive partitioning algorithm}
\label{algo: adaptive partitioning}
\begin{algorithmic}
\REQUIRE for robot $i$: its desired utilization rate $a_i$, and the function $f$ such that $c(z,g_i) = f( \|z-g_i\|)$ in (\ref{eq: Monge problem}).
\STATE \textbf{Initialization}: for $i \in [n]$, $w_{i} \gets 0$.
\STATE When event $k \geq 1$ appears at location $Z_k$:
\begin{itemize}
\STATE Run the \texttt{FloodMin} algorithm starting with $\hat d_j = f(\|Z_k - g_j\|) - w_{j}, j \in [n]$. 
\IF {robot $i$ terminates  \texttt{FloodMin} with $\mathtt{d_i} = \hat d_i$}
\STATE {$w_{i} \gets w_{i} + \gamma_{k-1} (a_i - 1)$}
\ELSE 
\STATE $w_{i} \leftarrow w_i + \gamma_{k-1} a_i$.
\ENDIF
\end{itemize}
\end{algorithmic}
\end{algorithm}

\begin{figure}[ht]
\centering
\includegraphics[width=\linewidth]{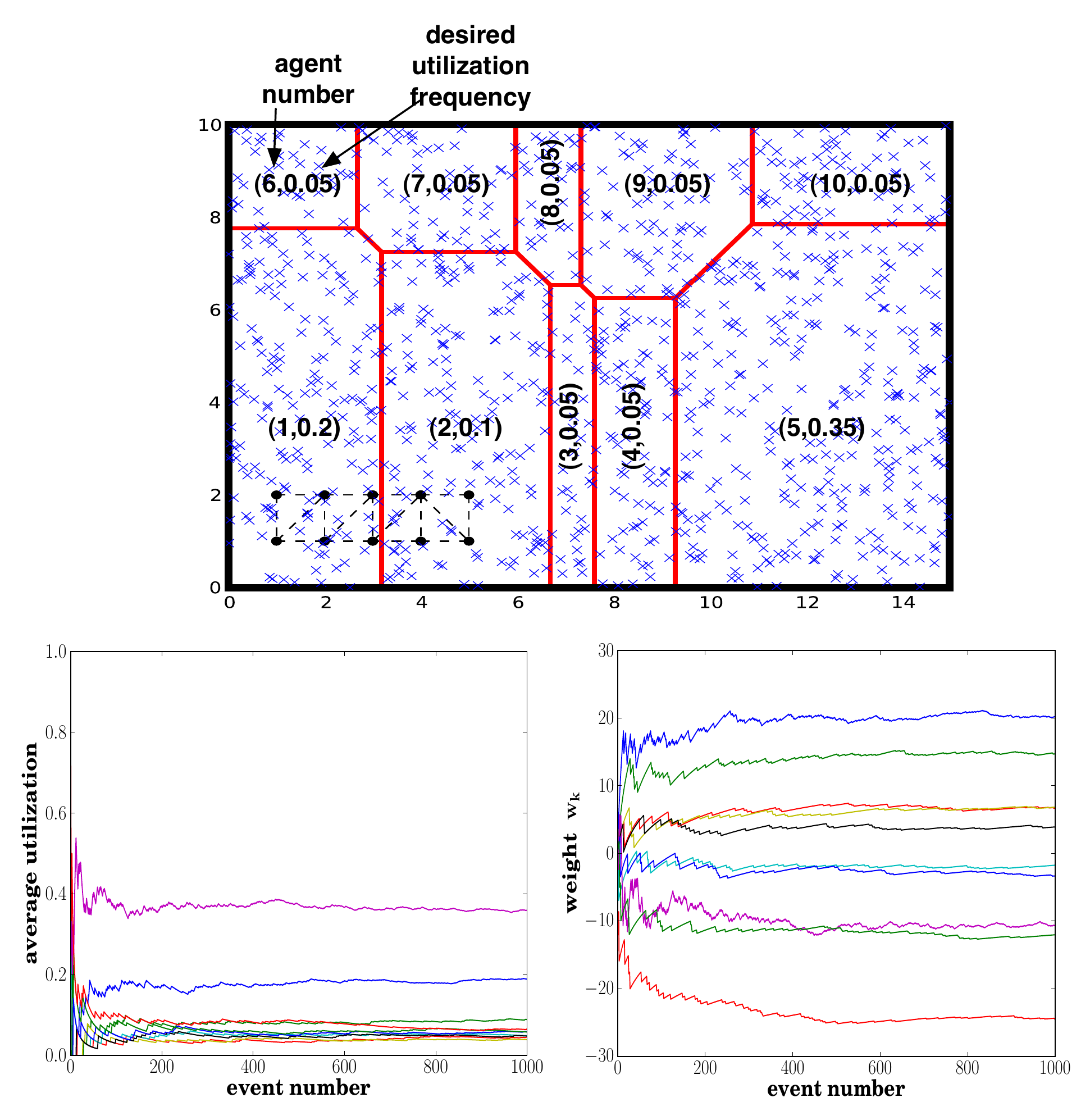}
\caption{Partition for $10$ robots after $1000$ events for the 
quadratic cost $c(z,g_i)=\|z-g_i\|^2$. The partition at each step is a power diagram. 
The desired utilization rates are shown for each agent on the figure. 
The power diagram generators used are represented as black dots in the lower left corner. 
Note that fixing their positions determines the orientation of the cell boundaries \cite{Aurenhammer91_VoronoiSurvey}. 
The power cells shown in red are computed using CGAL \cite{cgal}, 
but need not be computed by the agents running Algorithm \ref{algo: adaptive partitioning}. 
The bottom left figure shows the evolution of the empirical utilization frequencies over the first $1000$ events, 
and the bottom right figure the evolution of the weight vector $w_k$. 
The chosen stepsizes were $\gamma_k = 10/(1+0.01 k)$.}
\label{fig: partitioning}
\end{figure}


\section{Adaptive Dynamic Vehicle Routing}	\label{section: vehicle routing problems}

We now combine the algorithms of Section \ref{subsection: stochastic coverage} 
and Section \ref{section: load-balancing} to design an adaptive algorithm for 
the Dynamic Traveling Repairman Problem (DTRP). 
Assume for simplicity in this section that the environment is planar, i.e., $q=2$. 
In the DTRP \cite{Bertsimas91_DTRP}, events appear in the workspace $\mathsf Q$ 
according to a space-time Poisson process with rate $\lambda$ and spatial distribution $\mathbb P_z$. 
We assume as in Section \ref{subsection: stochastic coverage}  that $\mathsf Q$ is convex and compact.
When the $k^{th}$ event appears at time $t_k$, a robot needs to 
travel to its location $Z_k$ to service it. 
The robots travel at velocity $v$ according to the kinematic model (\ref{eq: trivial dynamics continuous time}).
The time that the $k^{th}$ event spends waiting for a robot to arrive at its location 
is denoted $W_k$. The robot then spends a random service time $S_k$ at the event location, 
where the variables $S_k$ are iid with finite first and second moments $\bar s, \overline {s^2}$. 
The system time of event $k$ is defined as $\Sigma_k = W_k + S_k, k \geq 1$.
The goal is to design policies for the robots that minimize the steady-state system time 
of the events $\overline \Sigma = \lim \sup_{k \to \infty} E[\Sigma_k]$. 
Let $\rho = \lambda \bar s / n$ denote the load factor, i.e., the average fraction of time a robot 
spends in on-site service \cite{Bertsimas93_DTRP}. 
Policies for the DTRP are usually analyzed in two limiting regimes, 
namely in light traffic conditions ($\lambda \to 0^+$) and heavy traffic conditions ($\rho \to 1^-$).

The policies for the DTRP initially proposed in 
\cite{Bertsimas91_DTRP, Bertsimas93_multiDTRP, Bertsimas93_DTRP} 
require the knowledge of the event distribution $\mathbb P_z$. 
The recent references \cite{Arsie09_routing, Pavone10_adaptiveDTRP} propose 
algorithms for the DTRP that work without this knowledge in the light traffic regime,
 but leave open the adaptive problem in heavy traffic.
The following sections make two contributions to the DTRP. 
First, in the light traffic case, we use the adaptive coverage control algorithm 
of Section \ref{subsection: stochastic coverage} to obtain an adaptive policy that is 
simpler than the solutions proposed so far \cite{Arsie09_routing, Pavone10_adaptiveDTRP} 
and provides the same convergence guarantees.
Second, for the heavy traffic case, we present the first fully adaptive policy
for the DTRP that provably stabilizes the system as long as it is stabilizable, 
in the absence of knowledge of $\mathbb P_z$. 
This policy relies on the adaptive partitioning algorithm of Section \ref{section: load-balancing}.

\subsection{Light Traffic Regime} 	\label{section: light traffic optimal policy}
Note first that we always have \cite{Bertsimas93_multiDTRP}
\begin{equation}	\label{eq: light traffic lower bound}
\overline \Sigma \geq \min_{p} \mathcal E_n(p) + \bar s,
\end{equation}
where $\mathcal E_n(p)$ is defined by (\ref{eq: coverage simple general}) for $f(x) = x/v$. 
This bound is tight in light traffic conditions \cite{Bertsimas91_DTRP, Bertsimas93_DTRP}, 
and achieved by the following policy. Let $p^*=[p_1^*,\ldots,p_n^*] \in \mathsf Q^{n}$ 
denote a global minimizer of $\mathcal E_n$, called a \emph{multi-median configuration}. 
In the absence of events, vehicle $i$ waits at the reference position $p_i^*$. 
When an event occurs, the agent whose reference position is closest to the event location services it. 
It then travels back to its reference position $p_i^*$.
As $\lambda \to 0^+$,  
the agents are at their 
reference configuration $p^*$ when a new event occurs, 
and this policy achieves the bound (\ref{eq: light traffic lower bound}) \cite{Bertsimas93_DTRP}.

To obtain an adaptive version of the above policy, we can use the
coverage control algorithm of Section \ref{subsection: stochastic coverage} 
to find a local minimizer of $\mathcal E_n$.
In the absence of an event, each robot waits at its 
current reference position $p_{i,k}$. When the $k^{th}$ event occurs at $Z_{k}$, 
the robot whose current reference position is closest to $Z_k$, say robot $j$, 
services the event, updates its reference position 
to $p_{j,k} = \Pi_\mathsf Q \left[ p_{j,k-1} + \gamma_k \frac{1}{v} \frac{Z_{k}-p_{j,k-1}}{\|Z_{k}-p_{j,k-1}\|}\right]$, 
and travels back toward $p_{j,k}$. 
Reasoning as in \cite{Bertsimas91_DTRP, Bertsimas93_DTRP},  
in the light traffic case where $\lambda \to 0$, the agents are 
at their reference positions when a new event occurs.
Hence the resulting policy achieves a steady-state system time 
of $\mathcal E_n(\hat p) + \bar s$, where $\hat p$ is a critical point of $\mathcal E_n$ 
to which the stochastic gradient algorithm (\ref{eq: update law Voronoi case}) 
converges under the assumptions of Theorem \ref{thm: main result coverage}.
For $n=1$, it achieves the minimum system time since $\mathcal E_1$ is convex. 
A similar guarantee is provided by the adaptive light traffic policy described 
in \cite{Arsie09_routing}, at the expense of a significantly more complex algorithm
where the robots keep track of all past locations visited.
Note that these policies turn out to be unstable as the load factor 
$\rho$ increases \cite{Bertsimas93_DTRP}, which motivates the heavy traffic
policy of the next section.

\subsection{A Stabilizing Adaptive Policy}

Policies adequate for the heavy-traffic regime but requiring $\mathbb P_z$ to be known 
are described in, e.g., \cite{Bertsimas93_DTRP, Xu95_thesis, Papastavrou96_DTRP, Pavone10_adaptiveDTRP}. 
The following non-adaptive policy, although not the best available, 
stabilizes the system in heavy-traffic, i.e., as $\rho \to 1^-$ \cite{Xu95_thesis, Papastavrou96_DTRP, Pavone10_adaptiveDTRP}.
We partition the workspace $\mathsf Q$ into $n$ regions $\{R_i\}_{i \in [n]}$ 
such that $\mathbb P_z(R_i) = 1/n, i \in [n]$. 
Robot $i$ only services the events occurring in region $R_i$. 
It does so by forming successive traveling salesman tours (TSP tours) 
through the event locations falling in this region, and servicing the events in the order of the tours. 
Recall that a TSP tour through a set of points is the shortest (here, for the Euclidean distance) 
closed tour visiting each point in the set once.
While robot $i$ services the events in a given tour, new events can occur in region $R_i$ 
and are backlogged by the robot. Once a tour is finished, the robot forms a new tour 
through the backlogged events and starts servicing them. 
Assuming that $\mathbb P_z$ has a density $\phi_z$, 
it is known that this policy achieves the following bounds on the system time 
in heavy-traffic \cite[theorems 4.2, 6.4]{Pavone10_adaptiveDTRP}
\begin{align}	
&\frac{C^*}{n^2} \leq \lim_{\rho \to 1^-} (1-\rho)^2 \overline \Sigma \leq \frac{2 C^*}{n}, 
\label{eq: system time bounds} \\
&\text{where } C^* = C \frac{\lambda \Big ( \int_{\mathsf Q} \phi_z(z)^{1/2} dz \Big )^2}{v^2} 
\text{ and } C \approx 0.253.	\nonumber
\end{align}
In addition, the right-hand side of (\ref{eq: system time bounds}) 
can be changed to $2C^*/n^2$ if $\mathbb P_z$ is the 
uniform distribution on $\mathsf Q$ \cite{Pavone10_adaptiveDTRP}.
Now consider the adaptive version of this policy described in Algorithm \ref{algo: adaptive DTRP},
which partitions the workspace as in Section \ref{section: load-balancing}, and works without
the knowledge of any event process parameter 
such as $\lambda$ or $\mathbb P_z$.
\begin{thm}
The adaptive policy of Algorithm \ref{algo: adaptive DTRP}
stabilizes the system as long as $\rho < 1$ and
achieves a steady-state system time 
satisfying the heavy traffic performance bound (\ref{eq: system time bounds}). 
Moreover if $n=1$, this adaptive policy is also optimal in the light traffic regime $\lambda \to 0^+$.
\end{thm}

\begin{IEEEproof}
As $\rho \to 1$, 
the region of each robot is never empty 
and hence the robot never enters the mode in Algorithm \ref{algo: adaptive DTRP} 
where it goes toward its reference position $p_i$ \cite{Pavone10_adaptiveDTRP}.
By Theorem \ref{thm: load balancing convergence}, the space partition  
allocating events to robots converges as $k \to \infty$ to a power diagram $\{R_{i}\}_{i \in [n]}$ such that $\mathbb P_z(R_i)=1/n$. 
Hence the adaptive policy behaves in steady-state as the non-adaptive policy 
and satisfies (\ref{eq: system time bounds}). 
In the light traffic regime ($\lambda \to 0$) and in steady-state,
each agent following Algorithm \ref{algo: adaptive DTRP} is at the median of its region $R_i$ 
when a new event occurs. 
Indeed the updates of the reference position $p_i$ can be viewed as a stochastic gradient descent algorithm
for the cost $g_i(p_i) = E[\|p_i - Z\| | Z \in R_i]$ (notice here that the space partition 
evolves independently of the reference locations $p_i, i \in [n]$, in constrast to Section \ref{subsection: stochastic coverage}).
In particular if $n=1$, there is just one cell and $\mathcal E_1$ 
is strictly convex, so the policy achieves the performance bound (\ref{eq: light traffic lower bound}).
\end{IEEEproof}

\begin{algorithm}
\caption{Adaptive DTRP algorithm.
Robot $i$ updates a weight $w_i \in \mathbb R$ as in Section \ref{section: load-balancing}, 
a reference position $p_{i} \in \mathbb R^q$,
and two sets of event locations $\mathcal O_i$ and $\mathcal P_i$. 
It is also associated to a fixed point $g_i \in \mathsf Q$, with $g_i \neq g_j$, for $i \neq j$. 
}
\label{algo: adaptive DTRP}
\begin{algorithmic}
\STATE \textbf{Initialization}: for $i \in [n]$, $w_{i} \gets 0$, 
$p_i \gets$ robot $i$'s initial position, 
$\mathcal O_i \gets \emptyset$, 
$\mathcal P_i \gets \emptyset$.
\STATE When event $k \geq 1$ appears at location $Z_k$:
\STATE \begin{itemize}
\item Run the \texttt{FloodMin} algorithm starting with $\hat d_j = \|Z_k - g_j\|^2 - w_{j}, j \in [n]$. 
\IF {robot $i$ terminates \texttt{FloodMin} with $\mathtt{d_i} = \hat d_i$}
\STATE $w_{i} \leftarrow w_{i} + \gamma_{k-1} (n - 1)/n$, 
$p_i \leftarrow \Pi_\mathsf Q \left [p_{i} + \gamma_{k-1} \frac{Z_{k}-p_{i}}{\|Z_{k}-p_{i}\|} \right]$,
$\mathcal O_i \gets \mathcal O_i \cup \{Z_k\}$.
\ELSE
\STATE $w_{i} \leftarrow w_i + \gamma_{k-1}/n$, and $p_i, \mathcal O_i$ remain unchanged.
\ENDIF
\end{itemize}
\STATE In parallel, execute the following process forever for each robot $i \in [n]$
\begin{enumerate}
\label{algo 1}
\STATE When $\mathcal P_i = \mathcal O_i = \emptyset$, robot $i$ travels toward $p_i$ and stays there if $p_i$ is reached.
\STATE When $\mathcal P_i = \emptyset$ and $\mathcal O_i \neq \emptyset$, then let $\mathcal P_i \gets \mathcal O_i$, $\mathcal O_i \gets \emptyset$.
Compute an Euclidean TSP tour through the points $\mathcal P_i$.
\STATE When $\mathcal P_i \neq \emptyset$, then service the locations in $\mathcal P_i$ in the order of the TSP tour, 
removing them from $\mathcal P_i$ when they are serviced. 
At the end of the tour, we are back in the situation $\mathcal P_i = \emptyset$. If $\mathcal O_i = \emptyset$, go to 1), otherwise, go to 2). 
\end{enumerate}
\end{algorithmic}
\end{algorithm}

\section{Conclusions}	\label{section: conclusions}

We have discussed robot deployment algorithms for coverage control, spatial partitioning and dynamic vehicle routing problems 
in the situation where the event location distribution is a priori unknown. By adopting the unifying point of view of stochastic gradient algorithms 
we can derive simple algorithms in each case that locally optimize the objective function (globally in the case of the partitioning problem). 
The coverage control and space partitioning algorithms are combined to provide a fully adaptive solution to the DTRP, 
with performance guarantees in heavy and light traffic conditions. 

Among the issues associated with stochastic gradient algorithms, we point out that they can be slower than their deterministic counterparts 
and that their practical performance is sensitive to the tuning of the stepsizes $\gamma_k$. Many guidelines are available 
in the literature on stochastic approximation algorithms for the selection of good stepsizes and possibly iterate averaging, see e.g. \cite{Kushner03_SA, Spall03_SA}. 
In addition, if some prior knowledge about the event distribution is available, it can be leveraged in a straightforward hybrid solution that 
first deploys the robots using a deterministic gradient algorithm. 
Once the robots have converged, the adaptive algorithm is used to correct for the modeling errors and environmental uncertainty, 
exploiting actual observations. 
Note also that the stochastic gradient algorithms can still be used if the distribution $\mathbb P_z$ is known, 
by generating random targets artificially and essentially evaluating integrals such as (\ref{eq: partial derivative distortion}) 
by Monte-Carlo simulations \cite{Pages98_quantization}. 
However, this method is generally only advantageous for dimensions $q$ sufficiently large.



%

\appendices
%


\section{Convergence of the Coverage Control Algorithm}	\label{appendix: Lloyd gradient flow}

In this appendix we collect a number of useful properties of the gradient system 
\begin{equation}	\label{eq: Lloyd ODE}
\dot p = - \nabla \mathcal E_n (p), \;\; p(0) \in \mathsf Q^n \setminus \mathsf D_n,
\end{equation}
where the distortion function $\mathcal E_n$ is defined in (\ref{eq: coverage simple general}),
and $\mathsf Q \subset \mathbb R^q$ is convex and compact.
As discussed below, this ODE is well defined on $\mathsf Q^n \setminus \mathsf D_n$.
We also consider its extension to $\mathsf Q^n$ in the form of the differential inclusion
\begin{equation}		\label{eq: Lloyd DI}
\dot p \in \mathcal H_n(p), \;\; p(0) \in \mathsf Q^n,
\end{equation}
where the the set-valued map $\mathcal H_n$ is defined in (\ref{eq: set value map gradient - specific}).
Note that $\nabla \mathcal E_n$ is piecewise continuous and $\mathcal H_n$ can in
fact equivalently be defined as \cite[p.51]{Cortes08_discontinuous}
\begin{equation}	\label{eq: set value map gradient pw cont.}
\mathcal H_n(p) 
= \begin{cases}
\{-\nabla \mathcal E_n(p)\} \;\; \text{if } p \notin \mathsf D_n, \\
\overline{\text{co}} \left \{ \lim_{k \to \infty} (- \nabla \mathcal E_n(p_k)) 
| p_k \to p \text{ as } k \to \infty \right \} \\
\hspace{5cm} \text{if } p \in \mathsf D_n.
\end{cases}
\end{equation}

Following the ODE method \cite{Ljung77_SA}, we can characterize the asymptotic behavior 
of the algorithm  (\ref{eq: update law Voronoi case}) 
as in Theorems \ref{thm: standard thm ODE}  and \ref{thm: main result coverage} 
by studying the properties of these continuous-time dynamical systems.
We assume as in section \ref{subsection: stochastic coverage} that $f: \mathbb R_{\geq 0} \to \mathbb R_{\geq 0}$ 
is increasing and continuously differentiable. 
We refer the reader to \cite{Pages98_quantization, Du99_Voronoi, Okabe2000_Voronoi, Cortes05_coverageComm} 
for previous work on the gradient system (\ref{eq: Lloyd ODE}).
In particular, \cite{Pages98_quantization} provides some convergence results for algorithm (\ref{eq: update law Voronoi case}),
and points out that the non-differentiability of $\mathcal E_n$ creates technical difficulties in the convergence proofs. 
We handle these issues by initially considering the differential inclusion (\ref{eq: Lloyd DI}) instead of the ODE (\ref{eq: Lloyd ODE}). 
%


\subsection{Differentiability Properties of $\mathcal E_n$}

The first task is to prove the Lipschitz continuity and differentiability properties of the function $\mathcal E_n$ 
stated in Proposition \ref{prop: derivative of distortion}. We follow the argument of 
\cite[Proposition 9]{Pages98_quantization}. Let us begin with some preliminary lemmas.

\begin{lem}	\label{lem: Lipschitz integrand}
For every $z \in \mathsf Q$, the function $p \to e(p,z) := \min_{i \in [n]} \{f(\|z-p_i\|\}$ is uniformly Lipschitz 
continuous, with
\begin{align*}
|e(p,z) - e(p',z)| \leq \left (\max_{x \in [0,\text{\emph{diam}}(\mathsf Q)]} f'(x) \right) \max_{i \in [n]} \|p_i-p'_i\|,\\
 \;\; \forall p, p' \in \mathsf Q^n, \; \forall z \in \mathsf Q.
\end{align*}
\end{lem}
\begin{IEEEproof}
Let $p, p' \in \mathsf Q^n$. Denote $\mathcal S = \{p_1,\ldots,p_n\} \subset \mathsf Q$, and similarly for 
$\mathcal S'$. We have $e(p,z) = f(d(z,\mathcal S))$. Then 
\begin{align*}
&| f(d(z,\mathcal S)) - f(d(z,\mathcal S')) | \\
&\leq \left (\max_{x \in [0,\text{diam}(\mathsf Q)]} f'(x) \right) | \min_{i \in [n]} \|p_i-z\| - \min_{j \in [n]} \|p'_j-z\| |.
\end{align*}
Fix $z \in \mathsf Q$, and denote $\min_{i \in [n]} \|p_i-z\| = \|p_{i_0} - z\|$ and 
$\min_{j \in [n]} \|p'_j-z\| = \|p'_{j_0} - z\|$. Then
\begin{align*}
\min_{i \in [n]} \|p_i-z\| - \min_{j \in [n]} \|p'_j-z\| \leq \|p_{j_0} - z\| - \|p'_{j_0} - z\| \\
\leq \|p_{j_0} - p'_{j_0} \| \leq \max_i \|p_i - p'_i\|,
\end{align*}
and
\begin{align*}
\min_{j \in [n]} \|p'_j-z\| - \min_{i \in [n]} \|p_i-z\| \leq \|p'_{i_0} - z\| - \|p_{i_0} - z\| \\
\leq \|p_{i_0} - p'_{i_0} \| \leq \max_i \|p_i - p'_i\|.
\end{align*}
This proves the lemma.
\end{IEEEproof}

\begin{lem} \label{eq: lemma derivative under integral}
Let $e: \mathbb R^m \times \mathsf Q \to \mathbb R$ be a function and $O \subset \mathbb R^{m}$ be an open set such that, for all $p \in O$, for all $i \in [n]$, the partial derivative $z \to \partial e/\partial p_i(p,z)$ exists $\mathbb P_z$-almost surely. Moreover, assume that there is a constant $C$ and a norm $\| \cdot \|$ on $\mathbb R^m$ such that $p \to e(p,z)$ is uniformly Lipschitz, i.e.,
\[
| e(p,z) - e(p,z) | \leq C \|p - p' \|, \;\; \forall p, p' \in O, \; \forall z \in \mathsf Q.
\]
Then the function $p \to \int_{\mathsf Q} e(p,z) \mathbb P_z(dz)$ is differentiable on $O$ and we have, for any $p \in O$, and for each $i \in [n]$,
\[
\frac{\partial}{\partial p_i} \int_{\mathsf Q} e(p,z) \mathbb P_z(dz) = \int_{\mathsf Q} \frac{\partial}{\partial p_i} e(p,z)  \mathbb P_z(dz).
\]
\end{lem}
\begin{IEEEproof}
Denote $F(p) =  \int_{\mathsf Q} e(p,z) \mathbb P_z(dz)$, and $\tilde e(p,h,z) = \frac{[e(p+h e_i,z) - e(p,z)]}{h}$. Then
\[
\lim_{h \to 0} \frac{F(p+he_i) - F(p)}{h} = \lim_{h \to 0} \int_{\mathsf Q} \tilde e(p,h,z) \mathbb P_z(dz).
\]
Now for all $h$ sufficiently small and all $z \in \mathsf Q$, we have $|\tilde e(p,h,z)| \leq C$ by the Lipschitz continuity assumption. Hence by the dominated convergence theorem, 
\begin{align*}
\frac{\partial}{\partial p_i} \int_{\mathsf Q} e(p,z) \mathbb P_z(dz) &= \lim_{h \to 0} \frac{F(p+he_i) - F(p)}{h} \\
&= \lim_{h \to 0} \int_{\mathsf Q} \tilde e(p,h,z) \mathbb P_z(dz) \\
&= \int_{\mathsf Q} \lim_{h \to 0} \tilde e(p,h,z) \mathbb P_z(dz)  \\
&= \int_{\mathsf Q} \frac{\partial}{\partial p_i} e(p,z)  \mathbb P_z(dz).
\end{align*}
\end{IEEEproof}

We now prove Proposition \ref{prop: derivative of distortion}.
\begin{IEEEproof}
[Proof of Proposition \ref{prop: derivative of distortion}]
The inequality
\[
| \mathcal E_n(p) - \mathcal E_n(p') | \leq \left (\max_{x \in [0,\text{diam}(\mathsf Q)]} f'(x) \right) \max_{i \in [n]} \|p_i-p'_i\|,
\]
is immediate from Lemma \ref{lem: Lipschitz integrand}, so $\mathcal E_n$ is globally Lipschitz continuous on $\mathsf Q^n$.
Assume now that $p \in \mathsf Q^n \setminus \mathsf D_n$.
Note that for $z \notin \cup_{i=1}^n \partial V_i(p)$ and $z \neq p_i, i \in [n]$, we have
\[
\frac{\partial}{\partial p_i} e(p,z) = f'(\|p_i - z\|) \frac{p_i - z}{\|p_i - z\|} 1_{V_i(p)} (z).
\]
Since the set where this formula does not hold has $\mathbb P_z$-mesure zero under our assumption (recall that $\partial V_i(p)$ consists of subsets of hyperplanes), $\mathcal E_n$ is differentiable at $p$ and (\ref{eq: partial derivative distortion}) is a direct consequence of Lemma \ref{eq: lemma derivative under integral}. Moreover, $\mathcal E_n$ is in fact continuously differentiable at $p$ if and only if the partial derivatives $\partial \mathcal E_n(p) / \partial p_i$ are continuous at $p$, for all $i \in [n]$. 
Since we assume that $f'$ is continuous on $\mathsf Q$, the function $p \to I(p,z) := f'(\|p_i - z\|) \frac{p_i - z}{\|p_i - z\|} 1_{V_i(p)} (z)$ is continuous on $\mathsf Q^n \setminus \mathsf D_n$ for $\mathbb P_z$-almost all $z \in \mathsf Q$ (i.e., for $z \notin \cup_{i=1}^n \partial V_i(p)$ and $z \neq p_i, i \in [n]$), and moreover $z \to I(p,z)$ is bounded on the compact set $\mathsf Q$. Hence the continuity of the partial derivatives (\ref{eq: partial derivative distortion}) is a consequence of the Lebesgue dominated convergence theorem.
\end{IEEEproof}


From Proposition \ref{prop: derivative of distortion}, the function $\mathcal E_n$ is 
continuously differentiable on $\mathbb R^n \setminus \mathsf D_n$. 
In general however, $\nabla \mathcal E_n$ is discontinuous on the set $\mathsf D_n$, 
see Fig. \ref{fig: simple Lloyd vector field}. 
\begin{figure}
\centering
\includegraphics[height=5cm]{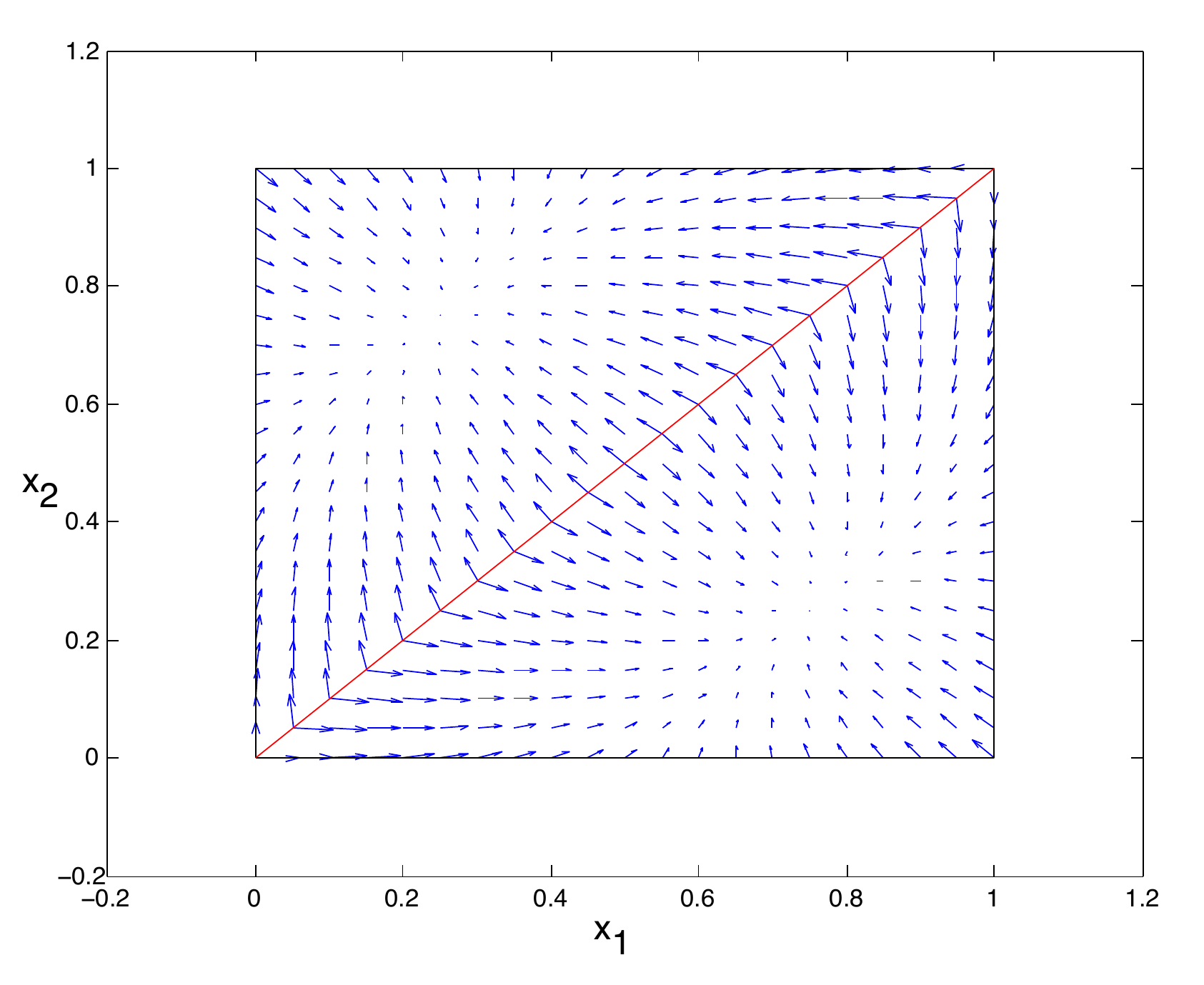}
\caption{Vector field for the gradient system (\ref{eq: Lloyd ODE}), with two agents evolving on $[0,1]$ 
and $\mathbb P_z$ uniform on $[0,1]$. The discontinuity on the line $x_1=x_2$ occurs when the two agents switch side, 
from $x_1 < x_2$ to $x_1 > x_2$. Note that the vector field is symmetric with respect to this line. 
The equilibrium occurs at a unique \emph{geometric} point on the line, namely $(1/4,3/4)$, 
corresponding to two stationary points for the flow, one for each ordering of the robots.
The differential inclusion (\ref{eq: Lloyd DI}) has an additional (unstable) equilibrium at $(1/2,1/2)$.
}
\label{fig: simple Lloyd vector field}
\end{figure}
%
%
To discuss more precisely the behavior of the gradient of $\mathcal E_n$ as we approach the set $\mathsf D_n$, define
\[
N(x) = \begin{cases}
\|\nabla \mathcal E_n(x)\|^2 & \text{ if } x \in \mathbb R^n \setminus \mathsf D_n \\
\lim \inf_{y \in \mathbb R^n \setminus \mathsf D_n, y \to x} \|\nabla \mathcal E_n(y)\|^2 & \text{ if } x \in \mathsf D_n.
\end{cases}
\]
Note that because $\nabla \mathcal E_n$ is continuous on $\mathbb R^n \setminus \mathsf D_n$  
the two definitions of $N$ in fact coincide on this set. 

Our goal is now the result of Proposition \ref{prop: nonzero gradient} below, whose proof follows that 
of \cite[Lemma 30]{Pages98_quantization}. We introduce first the notion of \emph{Voronoi aggregates}, 
which occur when several robot positions coincide. 
If $x = [x_1,\ldots,x_n] \in \mathsf D_n$, there is at least one set $J \subset [n], |J| \geq 2$, of components 
such that for all $i,j \in J, x_i = x_j=:x_{J}$ and for all $i \in J, j \notin J, x_i \neq x_j$. The set $J$ is called 
an aggregate of components of $x$, and $x \in \mathsf D_n$ can have several aggregates. The aggregate 
$J$ is then associated to the Voronoi cell $V_J(x) := \{ z \in \mathbb R^q \; | \; \| x_J - z \| \leq \| x_k - z \|, \forall k \in [n]\}$.
Let $p^{(k)} \in \mathsf Q^n \setminus \mathsf  D_n$ be a sequence converging to $p \in \mathsf  Q^n$. 
Consider the vectors 
\[
u_{ij}^{(k)} = \frac{p_j^{(k)} - p_i^{(k)}}{\|p_j^{(k)} - p_i^{(k)}\|}, \;\; i \neq j.
\]
Since the unit sphere is compact, we can extract a subsequence so that the unit vectors $u_{ij}^{(k)}$ converge, 
even if $i,j$ belong to an aggregate. We will need the following technical result.
\begin{lem}	\label{lem: technical set convergence}
Let $p^{(k)}$ be a converging sequence as above such that the vectors $u_{ij}^{(k)}$ converge, for all $i \neq j$. 
Denote $u_{ij} = \lim_{k \to \infty} u_{ij}^{(k)}$. Moreover, define for that sequence the sets
\begin{equation}	\label{eq: convergence of sets}
A_i := \{z \in \mathsf Q | \exists k_z \text{ s.t. } z \in V_i(p^{(k)}) \text{ for all } k \geq k_z \}, \; i \in [n].
\end{equation}
Let $z_0 \in \mathsf Q$. Then if hyperplanes in $\mathbb R^q$ have $\mathbb P_z$-measure zero, we have 
\begin{equation}	\label{eq: Voronoi cells convergence}
\mathbf 1_{V_i(p^{(k)})}(z_0) \to \mathbf 1_{A_i}(z_0) \text{ as } k \to \infty, \mathbb P_z\text{-almost everywhere.}
\end{equation}
\end{lem}

\begin{proof}
If $z_0 \in A_i$, then (\ref{eq: Voronoi cells convergence}) is clear. Otherwise, there exists (at least) two indices 
$i \neq j$ such that 
\begin{align}	\label{eq: alternating sets}
z_0 \in A_{ij} := \{z \in \mathsf Q | z \in V_i(p^{(k)}) \text{ and } z \in V_j(p^{(k')}) \nonumber \\
\text{ for infinitely many } k, k' \}.
\end{align}
Hence there are subsequences $l_1(k)$ and $l_2(k)$ such that
\begin{align}	\label{eq: extraction alternating subseq}
\|p_i^{l_1(k)} - z_0 \| \leq \|p_j^{l_1(k)} - z_0 \|, \;\; \forall k \geq 0, \nonumber \\
\text{ and } \|p_j^{l_2(k)} - z_0 \| \leq \|p_i^{l_2(k)} - z_0 \|, \;\; \forall k \geq 0.
\end{align}
Let us first assume that the indices $i$ and $j$ do not belong to an aggregate. 
Letting $k \to \infty$ in each subsequence, we obtain that $\|p_i - z_0 \| = \|p_j - z_0 \|$, hence $z_0$ 
belongs to the hyperplane perpendicularly bisecting the line segment $(p_i,p_j)$. But this hyperplane has 
$\mathbb P_z$-measure zero by assumption.
Now assume that $i, j$ belong to an aggregate $J$ and consider the vectors 
\[
u_{ij}^{(k)} = \frac{p_j^{(k)} - p_i^{(k)}}{\|p_j^{(k)} - p_i^{(k)}\|}.
\]
Then from (\ref{eq: extraction alternating subseq}) we get respectively
\begin{align*}
\ip{u_{ij}^{(l_1(k))}}{z_0-p_j^{(l_1(k))}} \leq 0, \forall k \geq 0, \\ 
\text{ and } \ip{u_{ij}^{(l_2(k))}}{z_0-p_i^{(l_2(k))}} \geq 0, \forall k \geq 0.
\end{align*}
Letting again $k \to \infty$ in the subsequences, and since we assumed that $\{u_{ij}^{(k)}\}$ converges, we obtain
\[
\ip{u_{ij}}{z_0-p_J} = 0,
\]
and so $z_0$ must belong to the hyperplane perpendicular to $u_{ij}$ passing through $p_J$. 
Again, this set has measure zero. Overall (\ref{eq: Voronoi cells convergence}) hold everywhere except 
perhaps on a finite number of hyperplanes, hence it holds $\mathbb P_z$-almost everywhere.
\end{proof}

\begin{prop}	\label{prop: nonzero gradient}
Suppose that Assumption \ref{assumption: hyperplanes} holds and that $\mathbb P_z$ dominates 
the Lebesgue measure on $\mathsf Q^n$. Then we have $N(p)>0$ for all $p \in \mathsf D_n \cap \mathsf Q^n$. 
Hence there exists $\delta_0>0$ such that 
\[
\inf_{p \in \mathsf Q^n \cap (B(\mathsf D_n \cap \mathsf Q^n,\delta_0) \setminus \mathsf D_n) } 
\|\nabla \mathcal E_n(p)\|^2 =: \kappa > 0.
\]
Moreover, there exists $B < \infty$ such that $\|y\|^2 \leq B$ for all $y \in \mathcal H_n(p)$ and all $p \in \mathsf Q^n$. 
\end{prop}

\begin{IEEEproof} 
Denote $\mathsf {\tilde D}_n := \mathsf D_n \cap \mathsf Q^n$.
Let $p \in \mathsf {\tilde D}_n$, and $J$ be an aggregate for $p$. By the characterization of lower limits, 
consider a sequence $p^{(k)} \in \mathsf Q^n \setminus \mathsf D_n$ converging to $p$ and 
such that $\|\nabla \mathcal E_n(p^k)\|^2 \to N(p)$ as $k \to \infty$. 
Up to taking a subsequence, we can assume that the sequences of unit vectors
$
u_{ij}^{(k)} := \frac{p_j^{(k)} - p_i^{(k)}}{\|p_j^{(k)} - p_i^{(k)}\|} 
$
converge for $i,j$ in the aggregate, so that
\[
u_{ij}^{(k)} \xrightarrow[k \to \infty]{} u_{ij}, \forall i \neq j.
\]
We are now in the situation of Lemma \ref{lem: technical set convergence}.
Next, consider the quantity
\begin{align*}
&I_{ij}^k = \\
&\int_{\mathsf Q} 1_{\{V_j(p^{(k)})\}}(z) f'(\|p^{(k)}_i - z\|) 
\underbrace{\ip{u_{ij}^{(k)}}{\frac{z - p^{(k)}_i}{\|z-p^{(k)}_i\|}}}_{\geq0} \mathbb P_z(dz),
\end{align*}
for $i,j \in J$. It is easy to see that for $z \in V_j(p^k)$, the inner product inside the integral is nonnegative. 
By the dominated convergence theorem, and using (\ref{eq: Voronoi cells convergence}), we have
\begin{equation}	\label{eq: def of Iij}
I_{ij}^k \xrightarrow[k \to \infty]{} I_{ij} := \int_{A_j} f'(\|p_J - z\|) 
\underbrace{\ip{u_{ij}}{\frac{z - p_J}{\|z-p_J\|}}}_{\geq 0} \mathbb P_z(dz).
\end{equation}
On the other hand, we also have by (\ref{eq: partial derivative distortion})
\begin{align*}
&\ip{u^{(k)}_{ij}}{- \frac{\partial \mathcal E_n}{\partial p_j}(p^{(k)})} = \\
&\int_{\mathsf Q} 1_{\{V_j(p^{(k)})\}}(z) f'(\|p^{(k)}_j - z\|) \ip{u_{ij}^{(k)}}{\frac{z - p^{(k)}_j}{\|z-p^{(k)}_j\|}} \mathbb P_z(dz),
\end{align*}
and so, 
\[
\lim_{k \to \infty} \ip{u^{(k)}_{ij}}{- \frac{\partial \mathcal E_n}{\partial p_j}(p^{(k)})} = 
\ip{u_{ij}}{\lim_{k \to \infty} - \frac{\partial \mathcal E_n}{\partial p_j}(p^{(k)})} = I_{ij},
\]
again by the dominated convergence theorem.
Hence if $\lim_{k \to \infty} - \frac{\partial \mathcal E_n}{\partial p_j}(p^k) = 0$, then $I_{ij} = 0$. 
But since the inner product inside the integral in (\ref{eq: def of Iij}) is zero only on the  hyperplane 
perpendicular to $u_{ij}$ passing through $p_J$, it must be strictly positive outside of this hyperplane. 
Also, we assumed that $f'$ is strictly positive. This implies that $\mathbb P_z(A_j) = 0$. 
This holds for all $j \in J$ if $N(p) = 0$.

But now take $z \in V_J$, the aggregate Voronoi cell. Then $z \in A_j$ for some $j \in J$, or $z \in A_{ij}$ 
as defined in (\ref{eq: alternating sets}), for some $i \neq j$. The set of points satisfying the second condition 
has measure $0$ as shown in Lemma \ref{lem: technical set convergence}. In other words, 
$V_j = \cup_{j \in J} A_j \cup \mathcal S$, with $\mathbb P_z (\mathcal S) = 0$. We obtain therefore 
$\mathbb P_z(V_J) = \sum_{j \in J} \mathbb P_z (A_j) = 0$. But this is impossible since $V_J$ is a polygon with 
non-empty interior, which has positive Lebesgue measure. Therefore $N(p) > 0$.

Next, recall the following definition of the lower limit \cite[p.8]{Rockafellar98_variational}
\[
\liminf_{x \to \bar x} f(x) = \sup_{V \in \mathcal N(\bar x)} \left [ \inf_{x \in V} f(x) \right],
\]
where $\mathcal N(x)$ denotes the set of neighborhoods of $x$.
Now if $N(p) > 0$ for all $p \in \mathsf {\tilde D}_n$, then by the characterization of the supremum, 
we have that for all $\bar p \in \mathsf {\tilde D}_n$, there exists a neighborhood 
$V(\bar p) \in \mathcal N(\bar p)$ such that
\[
0 < \frac{N(\bar p)}{2} \leq \inf_{p \in V(\bar p) \setminus \{\bar p\}} \|\nabla \mathcal E_n(p)\|^2 \leq N(\bar p).
\]
Without loss of generality, we can assume that $V(\bar p)$ is open. Then $\bigcup_{\bar p \in \mathsf {\tilde D}_n} V(\bar p)$ 
forms an open cover of the compact set $\mathsf {\tilde D}_n$, hence we can extract a finite subcover 
$V(\bar p^1), \ldots, V(\bar p^N)$. Clearly this finite subcover is again a neighorhood of $\mathsf {\tilde D}_n$, 
hence it contains $B(\mathsf {\tilde D}_n,\delta_0)$ for some $\delta_0 > 0$.
Since $N$ is a lower semi-continuous function, it attains it minimum $N^*$ on $\mathsf {\tilde D}_n$. 
We then have $\kappa \geq N^*/2 > 0$ in the proposition.

Finally, the fact the $N(p) < \infty$ follows from (\ref{eq: partial derivative distortion}). This immediately 
gives the last part of the proposition by definition of the set-valued map $\mathcal H_n$.
\end{IEEEproof}


\subsection{Trajectories of the Gradient System}

We now turn to the study of the trajectories of the ODE (\ref{eq: Lloyd ODE}) and the differential inclusion (\ref{eq: Lloyd DI}). 

\begin{prop}	\label{prop: finite time D intersection}
Suppose that Assumption \ref{assumption: hyperplanes} holds and that $\mathbb P_z$ 
dominates the Lebesgue measure on $\mathsf Q^n$.
If $x_0 \in \mathsf Q^n \setminus \mathsf D_n$, a trajectory $t \to x(t)$ of the ODE (\ref{eq: Lloyd ODE}) 
with $x(0) = x_0$ remains in $\mathsf Q^n \setminus \mathsf D_n$, i.e., 
for all $t < \infty$, $x(t) \in \mathsf Q^n \setminus \mathsf D_n$. 
Moreover, it converges to a compact connected subset of $\{x \in \mathsf Q^n \setminus \mathsf D_n: \nabla \mathcal E_n(x) = 0\}$.
Finally, a trajectory of the differential inclusion (\ref{eq: Lloyd DI}) starting from $x_0 \in \mathsf Q^n$ remains in $\mathsf Q^n$.
\end{prop}

\begin{IEEEproof}
First, consider the situation where a robot $i$ is on the boundary of the workspace, i.e., $p_i \in \partial \mathsf Q$. 
Assume first that the agents are separated, i.e., $p \in \mathsf Q^n \setminus \mathsf D_n$.
For a point $\hat x \in \partial \mathsf Q$, let us denote the normal cone at $\hat x$
\[
\mathcal C^n_{\mathsf Q}(\hat x) = \{ v \in \mathbb R^q | \ip{v}{x-\hat x} \leq 0, \forall x \in \mathsf Q \}.
\]
Now let $v \in \mathcal C^n_{\mathsf Q}(p_i)$. Then
\begin{align}	\label{eq: ineq normal cone}
\ip{v}{-\frac{\partial \mathcal E_n}{\partial p_i}(p)} &= \int_{V_i(p)} f'(\|p_i - z\|) \ip{v}{\frac{z - p_i}{\|z - p_i\|}} \mathbb P_z(dz) \nonumber \\
&\leq 0.
\end{align}
Hence the vector $-\partial \mathcal E_n / \partial p_i$ belongs to the polar of the normal cone, i.e., 
to the tangent cone of $\mathsf Q$ \cite[Cor. 5.2.5]{Hiriart-Urruty93}, and so the trajectories of the agents 
do not leave the workspace $\mathsf Q$. The inequality (\ref{eq: ineq normal cone}) is preserved by 
taking convex combinations, and holds at points $p \in \mathsf D_n$ for the differential inclusion (\ref{eq: Lloyd DI}) as well.

Next we show that a trajectory starting outside $\mathsf D_n$ never meet $\mathsf D_n$. 
Assume that for some $t^* \in (0,\infty)$, we have $p(t^*)=p^* \in \mathsf D_n$. Let $J$ be an aggregate 
of $p^*$, so that
\[
p^*_i=p^*_j =: p^*_J, \forall i,j \in J.
\]
Define, for all $i \neq j$ in $J$
\[
\phi_{ij}(t) = \|p_j(t)-p_i(t)\|, \;\; u_{ij}(t) = \frac{p_j(t)-p_i(t)}{\|p_j(t)-p_i(t)\|},
\]
and note that
\begin{align*}
\dot \phi_{ij}(t) &= \ip{u_{ij}(t)}{\dot p_j - \dot p_i} \\
&= \ip{u_{ij}(t)}{-\frac{\partial \mathcal E_n}{\partial p_j}(p(t)) 
+ \frac{\partial \mathcal E_n}{\partial p_i}(p(t))}.
\end{align*}
Next, consider the following quantity, for $i \neq j \in J$,
\begin{align*}
&I^1_{ij}(t) \\
&= \int_{\mathsf Q} 1_{V_j(p(t))}(z) f'(\|p_i - z\|) 
\underbrace{\ip{u_{ij}(t)}{\frac{z - p_i}{\|z-p_i\|}}}_{\geq 0} \mathbb P_z(dz),
\end{align*}
and note that because we integrate over the Voronoi cell of $j$, the inner product and thus $I_{ij}(t)$ 
are nonnegative, for all $t$. We then argue as in the proof of Proposition \ref{prop: nonzero gradient}.
Consider a sequence of times $\{t_k\}_k$ with $t_k \to t^*$ as $k \to \infty$. Up to taking a subsequence, 
we can assume $u_{ij}(t_k)$ converges to some unit vector $u_{ij}^*$.
By the dominated convergence theorem, defining $A_j$ as in (\ref{eq: convergence of sets}), we have
\begin{align*}
&\lim_{k \to \infty} I^1_{ij}(t_k) \\
&= \int_{\mathsf Q} 1_{A_j}(z) f'(\|p_J - z\|) \ip{u^*_{ij}}{\frac{z - p_J}{\|z-p_J\|}} \mathbb P_z(dz) \\
&= \lim_{k \to \infty} \ip{u_{ij}(t_k)}{-\frac{\partial \mathcal E_n}{\partial p_j}(p(t_k))}.
\end{align*}
where the last equality follows by another application of the dominated convergence theorem.
Similarly, defining
\begin{align*}
&I^2_{ij}(t) \\
&= \int_{\mathsf Q} 1_{V_i(p(t))}(z) f'(\|p_j - z\|) \underbrace{\ip{u_{ij}(t)}{\frac{p_j-z}{\|p_j-z\|}}}_{\geq 0} \mathbb P_z(dz),
\end{align*}
we have
\begin{align*}
&\lim_{k \to \infty} I^2_{ij}(t_k) \\
&= \int_{\mathsf Q} 1_{A_i}(z) f'(\|p_J - z\|) \ip{u^*_{ij}}{\frac{p_J-z}{\|p_J-z\|}} \mathbb P_z(dz) \\
&= \lim_{k \to \infty} \ip{u_{ij}(t_k)}{\frac{\partial \mathcal E_n}{\partial p_i}(p(t_k))}.
\end{align*}

As in the proof of Proposition \ref{prop: nonzero gradient}, because we assumed that $\mathbb P_z$ dominates 
the Lebesgue measure, one of the sets $A_i$ must have $\mathbb P_z(A_i) > 0$ since $\mathbb P_z(V_J) > 0$. 
This gives, with our assumption that hyperplanes have $\mathbb P_z$-measure zero,
\begin{align*}
\lim_{k \to \infty} \dot \phi_{ij}(t_k) \geq 0 \; \forall i,j, \text{ and } \lim_{k \to \infty} \dot \phi_{ij}(t_k) > 0 \\
\text{ for at least one pair $i,j$}.
\end{align*}
Thus there exists $i,j \in J$ such that $\liminf_{t \to t^*} \dot \phi_{ij}(t) > 0$. Therefore $\dot \phi_{ij}(t)$ is positive 
for $t \leq t^*$ close enough to $t^*$. But this contradicts the fact that $\phi_{ij}(t) = \|p_i(t)-p_j(t)\| \to 0$ as $t \to t^*$.

Finally, the convergence of the trajectories of the ODE to a compact connected invariant subset of 
$\{x \in \mathsf Q^n \setminus \mathsf D_n: \nabla \mathcal E_n(x) = 0\}$ follows from Lasalle's invariance principle.
\end{IEEEproof}

We can now show that the trajectories of the ODE never stay in $B(\mathsf D_n \cap \mathsf Q^n, \delta_0)$ for a long time. 
\begin{cor}	\label{eq: simple corollary exit from neighborhood of D}
Suppose that Assumption \ref{assumption: hyperplanes} holds and that $\mathbb P_z$ 
dominates the Lebesgue measure on $\mathsf Q^n$. 
Let $\delta_0>0, \kappa>0$ be defined as in Proposition \ref{prop: nonzero gradient}, 
$x_0 \in \mathsf Q^n \cap (B(\mathsf D_n \cap \mathsf Q^n,\delta_0) \setminus \mathsf D_n)$, and let 
$T = \frac{\max_{x \in \mathsf Q^n \cap B(\mathsf D_n \cap \mathsf Q^n,\delta_0)} \mathcal E_n(x)}{\kappa}$. 
Then a trajectory of the ODE passing through $x_0$ at time $t_1$ must exit $B(\mathsf D_n \cap \mathsf Q^n,\delta_0)$ 
at some time $t_2 \leq t_1 + T$.
\end{cor}

\begin{proof}
We have, for $t \geq t_1$ and as long as the trajectory $t \to x(t)$ remains in 
$B(\mathsf D_n \cap \mathsf Q^n,\delta_0) \setminus \mathsf D_n$
\begin{align*}
0 \leq \mathcal E_n(x(t)) &= \mathcal E_n(x_0) - \int_{t_1}^t \|\nabla \mathcal E_n(x(s))\|^2 ds \\
& \leq \max_{x \in \mathsf Q^n \cap (B(\mathsf D_n \cap \mathsf Q^n,\delta_0) \setminus \mathsf D_n) } \mathcal E_n(x) - \kappa (t-t_1).
\end{align*}
Hence the trajectory must exit $B(\mathsf D_n \cap \mathsf Q^n,\delta_0) \setminus \mathsf D_n$ at 
or before the time $t_2$ given in the theorem.
But we know by Proposition \ref{prop: finite time D intersection} that it cannot hit $\mathsf D_n$ at $t_2 < \infty$. 
Hence it must in fact exit $B(\mathsf D_n \cap \mathsf Q^n,\delta_0)$.
\end{proof}

The set $\mathsf C_n$ defined in (\ref{eq: critical points of distortion}) contains 
the set of limit points of the ODE (\ref{eq: Lloyd ODE}) by Proposition \ref{prop: finite time D intersection}.
From the definition of the set-valued map $\mathcal H_n$, the set $\mathcal L$ of limit points of 
the differential inclusion (\ref{eq: Lloyd DI}) consists of the set of limit points of the ODE (\ref{eq: Lloyd ODE}) 
together with the limit points of the sliding trajectories that start and remain on $\mathsf D_n$ 
(since a trajectory leaving $\mathsf D_n$ does not converge to $\mathsf D_n$ by Proposition \ref{prop: finite time D intersection}). 
Hence $\mathcal L \subset \mathsf C_n \cup (\mathsf D_n \cap \mathsf Q^n)$.
Moreover, we know by Proposition \ref{prop: nonzero gradient} 
that $\mathsf C_n \subset \mathsf Q^n \setminus B(\mathsf D_n,\delta_0)$ 
if $\mathbb P_z$ dominates the Lebesgue measure.


\subsection{Convergence of the Adaptive Coverage Control Algorithm}	\label{section: appdx pf coverage control}

We now prove the main convergence theorem for the adaptive coverage control algorithm.

\begin{IEEEproof}
[Proof of Theorem \ref{thm: main result coverage}]
First, for the proof of convergence, we can ignore the projection $\Pi_\mathsf Q$ in (\ref{eq: update law Voronoi case}).
In general, the analysis involves the corresponding projected ODE or projected differential inclusion, 
see \cite{Kushner03_SA}, \cite[chapter 5]{Borkar08_SA}. Note however from Proposition \ref{prop: finite time D intersection} 
that at any boundary point of $\mathsf Q^n$, the velocity vector of the unprojected differential inclusion 
is already in the tangent cone of $\mathsf Q^n$. 
Hence the projection step does not change the continuous-time dynamics and the convergence properties 
remain the same as for the unprojected algorithm. Moreover, the saturation function does not change 
the convergence properties either \cite[Section 1.3.5]{Kushner03_SA}.

Now the fact that with probability one, a sequence of iterates of (\ref{eq: update law Voronoi case}) 
converges to a compact connected invariant set of the differential inclusion (\ref{eq: Lloyd DI}) is standard, 
see, e.g., \cite[chapter 5]{Borkar08_SA}, \cite[Theorem 8.1 p. 195]{Kushner03_SA}. 
Consider a sample $\omega$ such that $\{p_k(\omega)\}$ converges to such a set, denoted $S$. 
In view of Proposition \ref{prop: finite time D intersection}, we have $S \subset \mathsf Q^n$.
Suppose that $S$ is not entirely contained in $\mathsf D_n$, and take $a \in S \setminus \mathsf D_n$. 
Then a trajectory of the differential inclusion passing through $a$ at $t=0$ is 
in fact a trajectory of the ODE (\ref{eq: Lloyd ODE}) for $t \geq 0$, by Proposition \ref{prop: finite time D intersection}. 
Because $S$ is invariant, we must then have $\dot {\mathcal E}_n(a) := -\| \nabla \mathcal E_n(a)\|^2 = 0$, 
i.e., $a \in \mathsf C_n$. This proves the first part of the theorem.

If $\mathbb P_z$ dominates the Lebesgue measure, then we know that $\mathsf C_n$ 
and $\mathsf D_n$ are disconnected by Proposition \ref{prop: nonzero gradient}, so $S$ is contained in one of these sets. 
Recall that under Assumption \ref{assumption: hyperplanes}, 
we can assume that almost surely, the iterates $\{p_k\}_{k \geq 0}$ of (\ref{eq: update law Voronoi case}) 
never hit $\mathsf D_n$. 
Choose the sample $\omega$ above in this set of probability $1$,
and recall the definitions of $\delta_0$ and $T$ from Corollary \ref{eq: simple corollary exit from neighborhood of D}. 
Suppose now that $S \subset \mathsf D_n$. Then there exists $k_0$ such that 
for all $k \geq k_0$, $p_k \in B(\mathsf D_n, \delta_0/4)$. For any $k \geq 0$, denote by $x^k(\cdot)$ 
the solution of the ODE (\ref{eq: Lloyd ODE}) starting at $p_k$ (i.e., $x^k(0) = p_k$). 
Also, denote by $\bar p$ the piecewise linear interpolation of the sequence $p_k$ with stepsizes $\gamma_k$.

Then by \cite[Chapter 2, Lemma 1]{Borkar08_SA}, there exists $k_1 \geq k_0$ such that for all $k \geq k_1$, 
we have $\sup_{t \in [t_k, t_k+T]} \|\bar p(t) - x^k(t)\| \leq \delta_0/4$, where $t_k := \sum_{l=0}^{k-1} \gamma_{l}$. 
In particular, $\|\bar p(t_k+T) - x^k(t_k+T)\| \leq \delta_0/4$. Now remark that by Corollary \ref{eq: simple corollary exit from neighborhood of D}, 
we must have $d(x^k(t_k+T),\mathsf D_n) > \delta_0$. By possibly increasing $k_1$, 
we can assume that there is an iterate $p_{\tilde k}$ with $\tilde k \geq k$ such that $\|p_{\tilde k} - \bar p(t_k+T)\| \leq \delta_0/4$. 
So we have $\| p_{\tilde k} - x^k(t_k + T)\| \leq \delta_0/2$, hence $d(p_{\tilde k}, \mathsf D_n) > \delta_0/2$. 
But this contradicts our assumptions that $p_{\tilde k} \in B(\mathsf D_n, \delta_0/4)$. 
Hence we cannot have $S \subset \mathsf D_n$ and so $S \subset \mathsf C_n$. 
This finishes the proof of the theorem. 

\end{IEEEproof}


\section{Space Partitioning and Optimal Transportation}	\label{appendix: optimal transport}

In this section we prove Theorem \ref{thm: partitioning duality theorem}, 
which forms the basis for the stochastic gradient Algorithm \ref{algo: adaptive partitioning}, 
partitioning the workspace between the agents. 
Compared to the results presented in the recent papers \cite{Pavone09_equipartitions, Cortes05_coverageComm}, 
this theorem makes weaker assumptions on the cost function $c(z,g)$ and on the target distribution $\mathbb P_z$. 
The main tool on which Theorem \ref{thm: partitioning duality theorem} relies 
is Kantorovich duality \cite{Rachev98_optimalTransport}. 
See also \cite{Abdellaoui93, Gangbo96_optimalTransport, Rueschendord00_Voronoi} for related results.

\begin{IEEEproof}[Proof of Theorem \ref{thm: partitioning duality theorem}]
We start by relaxing the optimization problem (\ref{eq: Monge problem}), (\ref{eq: partitioning constraint}) 
to the following Monge-Kantorovich optimal transportation Problem (MKP) \cite{Rachev98_optimalTransport}. 
Let $P_2 = \sum_{i=1}^n a_i \delta_{g_i}$, so that (\ref{eq: partitioning constraint}) 
can be rewritten $\mathbb P_z \circ T^{-1} = P_2$. We consider the minimization problem
\[
\min_{\pi \in \mathcal M(\mathbb P_z, P_2)} \int_{\mathsf Q \times \mathsf Q} c(z,g) d\pi(z,g),
\]
where $\mathcal M(\mathbb P_z, P_2)$ is the set of measures on $\mathsf Q \times \mathsf Q$ 
with marginals $\mathbb P_z$ and $P_2$, i.e., 
\[
\pi(A \times \mathsf Q) = \mathbb P_z(A), \;\; \pi(\mathsf Q \times B) = P_2(B),
\]
for all Borel subsets of $A,B$ of $\mathsf Q$. In other words, we are considering 
the problem of transferring some mass from locations distributed according to $\mathbb P_z$ 
to locations distributed according to $P_2$, and there is a cost $c(z,g)$ 
for moving a unit of mass from $z$ to $g$. Then $\pi$ is a transportation plan from the initial to the final locations, 
assuming that we allow a unit of mass to be split. The case where this splitting is not allowed, 
i.e., where we restrict $\pi$ to be of the form
\[
d\pi(z,g) = d\mathbb P_z(z) \delta_{T(z)}(g), 
\]
for some measurable function $T$, is a Monge Problem (MP) \cite{Monge1781_optimalTransport}, 
and is exactly our problem (\ref{eq: Monge problem}), (\ref{eq: partitioning constraint}). 
In our case where the target distribution $P_2$ is discrete, \cite[Theorem 3]{Cuesta93_discreteOptimalTransport} 
shows that solving the MKP gives a solution in the form of a transference \emph{function} $T$, 
i.e., a solution to the MP, under the assumption A\ref{assumption: regularity of Pz} of the theorem, 
and assuming the infimum in (\ref{eq: Monge problem}) is attained. 
This is the case if $c$ is lower bounded and lower semicontinuous and $\mathbb P_z$ 
is tight \cite[Remark 2.1.2]{Rachev98_optimalTransport}, and this last condition 
is satisfied since we assume $\mathsf Q$ compact.
Next, by Kantorovitch duality \cite{Rachev98_optimalTransport}, we have
\begin{align}	\label{eq: Kantorovitch duality equation}
& \min_{\pi \in \mathcal M(\mathbb P_z, P_2)} \int_{\mathsf Q \times \mathsf Q} c(z,g) d\pi(z,g) \\
& = \sup_{(\phi,w) \in \Phi_c } \left \{\int_{\mathsf Q} \phi(z) \; d \mathbb P_z(z) + \sum_{i=1}^n a_i w_i \right \}, \nonumber
\end{align}
where $\Phi_c$ is the set of pairs $(\phi,w)$ with $\phi: \mathsf Q \to \mathbb R$ 
in $L^1(\mathsf Q,\mathbb P_z)$, $w \in \mathbb R^n$, such that
\begin{equation}	
\phi(z) + w_i \leq c(z,g_i),
\end{equation}
for $\mathbb P_z$-almost all $z$ in $\mathsf Q$ and for all $i$ in $[n]$. 
Now for any $w \in \mathbb R$, define the function $w^c: \mathsf Q \to \mathbb R$ such that
\[
w^c(z) = \min_{i \in [n]} \{c(z,g_i) - w_i\}.
\]
From the definition of $\Phi_c$, we can then without loss of generality restrict 
the supremum on the right-hand side of (\ref{eq: Kantorovitch duality equation}) 
to pairs of the form $(w^c,w)$. Combining this with the previous remark on 
the Monge solution to the Monge-Kantorovitch problem, we get
\begin{align}	\label{eq: duality explicit}
& \min_{\substack{T: \mathsf Q \to \{g_1,\ldots,g_n\}\\ \mathbb P_z \circ T^{-1} = P_2}} \int_{\mathsf Q} c(z,T(z)) \mathbb P_z(dz) \\
& = \sup_{w \in \mathbb R^n} \left \{\int_{\mathsf Q} \min_{i \in [n]} \{c(z,g_i) - w_i\} 	\nonumber
\; \mathbb P_z(dz) + \sum_{i=1}^n a_i w_i \right \}.
\end{align}
Hence the value of the optimization problem is equal to the supremum of the function $h$ 
defined in (\ref{eq: dual function maximization}). The fact that the supremum is attained 
in the right hand side of (\ref{eq: duality explicit}) follows from, e.g., \cite[Theorem 2.3.12]{Rachev98_optimalTransport} 
under our assumption A\ref{eq: assumption on the cost} for $c$.

The function $h$ is concave since $w \to \min_{i \in [n]} \{c(z,g_i) - w_i\}$ is 
concave for all $z$ as the minimum of affine functions, and the integration 
with respect to $z$ preserves concavity. Finally, for $w^1, w^2 \in \mathbb R^n$, we have
\begin{align*}
h(w^2) - h(w^1) = \int_{\mathsf Q} \min_{i \in [n]} \{c(z,g_i) - w^2_i\} \; \mathbb P_z(dz) \\
- \int_{\mathsf Q} \min_{i \in [n]} \{c(z,g_i) - w^1_i\} \; \mathbb P_z(dz) + \sum_{i=1}^n a_i (w^2_i - w^1_i).
\end{align*}
Denoting $T^1$ an assignment that is optimal for $w^1$ (given by a generalized Voronoi partition), we have then, for all $z \in \mathsf Q$,
\[
\min_{i \in [n]} \{c(z,g_i) - w^2_i\} \leq c(z,T^1(z)) - w^2_i, 
\]
and so
\begin{align*}
h(w^2) - h(w^1) \leq - \sum_{i=1}^n \mathbb P_z(V^c_i(\mathcal G, w^1)) (w^2_i - w^1_i) \\
+ \sum_{i=1}^n a_i (w^2_i - w^1_i).
\end{align*}
But this inequality exactly says that $[a_1-\mathbb P_z(V^c_1(\mathcal G, w^1)), 
\ldots, a_n-\mathbb P_z(V^c_n(\mathcal G, w^1))]^T$ is a supergradient of $h$ at $w^1$.
For the convergence of the supergradient algorithm, see \cite[Proposition 8.2.6. p. 480]{Bertsekas03_convexBook}.
\end{IEEEproof}



\ifCLASSOPTIONcaptionsoff
  \newpage
\fi



\bibliographystyle{IEEEtran}
\bibliography{/Users/jerome/Dropbox/Research/bibtex/activeSensing,/Users/jerome/Dropbox/Research/bibtex/optimization,/Users/jerome/Dropbox/Research/bibtex/probability,/Users/jerome/Dropbox/Research/bibtex/controlSystems,/Users/jerome/Dropbox/Research/bibtex/communications}
\end{document}